\documentclass[final,hidelinks]{dmtcs-episciences}
\usepackage{amssymb, amsmath, amsthm, latexsym}
\usepackage{eucal}
\usepackage{ytableau}
\usepackage{undertilde}
\usepackage{cancel}
\usepackage{caption}
\usepackage[pdftex]{graphicx}

\newtheorem{thm}{Theorem}[section]
\newtheorem{cor}[thm]{Corollary}
\newtheorem{lem}[thm]{Lemma}
\newtheorem{prop}[thm]{Proposition}

\newtheorem{fact}[thm]{Fact}
\newtheorem{prob}[thm]{Problem}
\theoremstyle{definition}
\newtheorem{defn}[thm]{Definition}
\theoremstyle{remark}

\numberwithin{equation}{section}

\begin{document}

\title[Parabolic Catalan numbers count flagged Schur functions]{Parabolic Catalan numbers count flagged Schur functions and their appearances as type A Demazure characters (key polynomials)}

\author{Robert A. Proctor\affiliationmark{1} \and Matthew J. Willis\affiliationmark{2}}

\affiliation{University of North Carolina, Chapel Hill, NC 27599 U.S.A. \\
Wesleyan University, Middletown, CT 06457 U.S.A.}


\received{2017-6-20}

\revised{2017-11-10}

\accepted{2017-11-23}

\publicationdetails{19}{2017}{3}{15}{3727}

\maketitle

\begin{abstract}

Fix an integer partition $\lambda$ that has no more than $n$ parts.  Let $\beta$ be a weakly increasing $n$-tuple with entries from $\{1,...,n\}$.  The flagged Schur function indexed by $\lambda$ and $\beta$ is a polynomial generating function in $x_1, ..., x_n$ for certain semistandard tableaux of shape $\lambda$.  Let $\pi$ be an  $n$-permutation.  The type A Demazure character (key polynomial, Demazure polynomial) indexed by $\lambda$ and $\pi$ is another such polynomial generating function.  Reiner and Shimozono and then Postnikov and Stanley studied coincidences between these two families of polynomials.  Here their results are sharpened by the specification of unique representatives for the equivalence classes of indexes for both families of polynomials, extended by the consideration of more general $\beta$, and deepened by proving that the polynomial coincidences also hold at the level of the underlying tableau sets.  Let $R$ be the set of lengths of columns in the shape of $\lambda$ that are less than $n$.  Ordered set partitions of $\{1,...,n\}$ with block sizes determined by $R$, called $R$-permutations, are used to describe the minimal length representatives for the parabolic quotient of the  $n^{th}$ symmetric group specified by the set $\{1,...,n-1\} \backslash R$.  The notion of 312-avoidance is generalized from  $n$-permutations to these set partitions.  The $R$-parabolic Catalan number is defined to be the number of these.  Every flagged Schur function arises as a Demazure polynomial.  Those Demazure polynomials are precisely indexed by the $R$-312-avoiding $R$-permutations.  Hence the number of flagged Schur functions that are distinct as polynomials is shown to be the $R$-parabolic Catalan number.  The projecting and lifting processes that relate the notions of 312-avoidance and of $R$-312-avoidance are described with maps developed for other purposes.

\end{abstract}

\vspace{1pc}\noindent\textbf{MSC Codes.}  {05E15, 05A15, 05E10, 14M15}

\section{Introduction}

This is the second of three papers that develop and use structures which are counted by a ``parabolic'' generalization of Catalan numbers.  Apart from some motivating remarks, it can be read by anyone interested in tableaux or symmetric functions.  It is self-contained, except for some references to its tableau precursors.  Our predecessor paper \cite{PW3} listed roughly a half-dozen parabolic Catalan structures.  Theorem \ref{theorem18.1} here lists roughly a dozen more such structures;  it can be understood in conjunction with this introduction as soon as the definitions of its objects have been read.  Experimental combinatorialists may be interested in Problem \ref{prob14.5}.  Algebraic geometers may be interested in Problem \ref{prob16.1}.

Let  $n \geq 1$  and  $N \geq 1$.  Set  $[n] := \{1,2,...,n\}$.  Fix a partition  $\lambda$  of  $N$  that has no more than  $n$  parts.  Use the symbol  $\lambda$  to also denote the shape of this partition;  it has  $N$  boxes.  Let  $R \subseteq [n-1]$  be the set of lengths of columns in  $\lambda$  that are less than  $n$.  Set  $r := |R|$  and  $J := [n-1] \backslash R$.  Multipermutations of  $n$  non-distinct items whose repetition multiplicities are determined by  $R$  can be used to describe the minimum length representatives in  $W^J$  for the cosets in the quotient  $S_n/W_J$,  where  $S_n$  is the symmetric group and  $W_J$  is its parabolic subgroup specified by  $J$.  Given a semistandard tableau  $T$  of shape  $\lambda$  with values from  $[n]$,  let  $\Theta(T)$  denote the census for these values in  $T$.  Given variables  $x_1, ..., x_n$,  the weight of  $T$  is the monomial  $x^{\Theta(T)}$.  The Schur function  $s_\lambda(x)$  is the sum of the weight  $x^{\Theta(T)}$  over all semistandard tableaux  $T$  of shape $\lambda$.

Flagged Schur functions arose in 1982 when Lascoux and Sch\"{u}tzenberger were studying Schubert polynomials for the  $GL(n)/B$  flag manifold.  Let  $\beta$  be a weakly increasing  $n$-tuple with entries from  $[n]$.  Using the weight  $x^{\Theta(T)}$,  the flagged Schur function  $s_\lambda(\beta;x)$  is usually defined to be the polynomial generating function for the semistandard tableaux with values row-bounded by the entries of  $\beta$.  To exclude the zero polynomial,  we also require  $\beta_i \geq i$.  We call these ``flag Schur polynomials''.  Let  $UF_\lambda(n)$  denote the set of all such  $n$-tuples  $\beta$,  which we call ``upper flags''.  Demazure characters arose in 1974 when Demazure introduced certain  $B$-modules while studying singularities of Schubert varieties in the  $G/B$  flag manifold.  Let  $\pi$  be a permutation of  $[n]$.  From the work of Lascoux and Sch\"{u}tzenberger it is known that the type A Demazure character (key polynomial)  $d_\lambda(\pi;x)$  can be expressed as the polynomial generating function for certain semistandard tableaux specified by  $\pi$,  again using the weight  $x^{\Theta(T)}$.  We call these ``Demazure polynomials''.  Reiner and Shimozono \cite{RS} and then Postnikov and Stanley \cite{PS} studied coincidences between the  $s_\lambda(\beta;x)$  and the  $d_\lambda(\pi;x)$.  We sharpen their results by specifying unique representatives for the equivalence classes of indexes for both families of polynomials,  extend their results by considering more general  $\beta$,  and deepen their results by proving that the polynomial coincidences must also hold at the level of the underlying tableau sets.

Note that  $\lambda$  is strict exactly when  $R = [n-1]$.  Our results become easy or trivial when  $\lambda$  is strict,  and various aspects of our new structures become invisible when   $R = [n-1]$  as these structures reduce to familiar structures.  The number of  312-avoiding permutations of  $[n]$  is the  $n^{th}$ Catalan number  $C_n$;  from our general perspective this fact is occurring in the prototypical  $R= [n-1]$  case.

How many flag Schur polynomials are there,  up to polynomial equality?  We can answer this question via our finer study of the coincidences between flag Schur polynomials and Demazure polynomials,  which is conducted after we have first precisely indexed these families of polynomials.  Reiner and Shimozono showed that every flag Schur polynomial arises as a Demazure polynomial,  and Postnikov and Stanley then showed that the Demazure polynomials that participate in such coincidences can be indexed by the  312-avoiding permutations.  But one cannot deduce from this that the number of distinct flag Schur polynomials is  $C_n$,  since they ``loosely'' indexed their Demazure polynomials by permutations when  $R \subset [n-1]$.  It is not hard to see that the Demazure polynomials are precisely indexed by the cosets in  $W/W_J$.  We depict these cosets with standard forms for ordered set partitions of  $[n]$  whose  $r+1$  block sizes are determined by  $R$,  and refer to these standard forms as ``$R$-permutations''.  (The  $R$-permutations can be thought of as being ``inverses'' for the multipermutations mentioned above.)  While permutations precisely index the Schubert varieties in the full flag manifold  $GL(n)/B$,  one needs  $R$-permutations to precisely index the Schubert varieties in the general flag manifold  $GL(n)/P_J$.  (Here  $P_J$  is the parabolic subgroup  $BW_JB$  of  $GL(n)$.)  In \cite{PW3} we introduced a notion of  ``$R$-312-avoidance'' for the  $R$-permutations and defined the  $R$-parabolic Catalan number  $C_n^R$  to be the number of such  $R$-permutations.  Here we show that the Demazure polynomials that participate in coincidences are precisely indexed by the  $R$-312-avoiding  $R$-permutations.  This implies that the number of distinct flag Schur polynomials on the shape  $\lambda$  is  $C_n^R$.

As we were writing the earlier version \cite{PW2} of this paper, we learned that Godbole, Goyt, Herdan, and Pudwell had recently introduced \cite{GGHP} a notion of pattern avoidance by ordered partitions like ours,  but for general patterns.  They apparently developed this notion purely for enumerative motivations.  Chen, Dai, and Zhou soon obtained \cite{CDZ} further enumerative results for these pattern avoiding ordered partitions.  Our observations that  $C_n^R$  counts flag Schur polynomials (Theorem \ref{theorem18.1}(\ref{theorem18.1.4})) and convex Demazure tableau sets \cite{PW3} are apparently the first observances of these counts ``in nature''.  See Section 8 of \cite{PW3} for further remarks on \cite{GGHP} and \cite{CDZ}.

To obtain a complete picture for the coincidences,  it is necessary to develop precise indexes for the flag Schur polynomials as well.  When are two flag Schur polynomials to be regarded as being ``the same''?  Postnikov and Stanley apparently regarded two flag Schur polynomials  $s_\lambda(\beta;x)$  and  $s_{\lambda'}(\beta';x)$  to be the same only when  $\lambda = \lambda'$  and  $\beta = \beta'$:  On p. 158 of \cite{PS} it was stated that there are  $C_n$  flag Schur polynomials for any shape  $\lambda$.  For  $n = 3$  and  $\lambda = (1,1,0)$,  note that  $s_\lambda((3,3,3);x) = x_1x_2 + x_1x_3 + x_2x_3 = s_\lambda((2,3,3);x)$.  So for this  $\lambda$  there cannot be  $C_3 = 5$  flag Schur polynomials that are distinct as polynomials.

In addition to the notion of polynomial equality,  we consider a second way in which two of the polynomials considered here can be ``the same''.  Let  $P_\lambda$  and  $Q_{\lambda'}$  be sets of semistandard tableaux with values from  $[n]$  that are respectively of shape  $\lambda$  and of shape  $\lambda'$.  Using the weight  $x^{\Theta(T)}$,  let  $p_\lambda(x)$  and  $q_{\lambda'}(x)$  be the corresponding polynomial generating functions.  We define  $p_\lambda(x)$  and  $q_{\lambda'}(x)$  to be ``identical as generating functions'' if  $P_\lambda = Q_{\lambda'}$  as sets.  Then we write  $p_\lambda(x) \equiv q_{\lambda´}(x)$;  clearly we must have  $\lambda = \lambda'$  here.  If  $p_\lambda(x) = q_{\lambda'}(x)$  while  $p_\lambda(x) \hspace{1mm} \cancel{\equiv} \hspace{1mm} q_{\lambda'}(x)$,  we say that  $p_\lambda(x)$  and  $q_{\lambda'}(x)$  are ``accidentally equal''.  Denoting the underlying tableau sets for  $s_\lambda(\beta;x)$  and  $s_{\lambda'}(\beta';x)$  by  $\mathcal{S}_\lambda(\beta)$  and  $\mathcal{S}_{\lambda'}(\beta')$,  we write  $\beta \approx_\lambda \beta'$  when  $\mathcal{S}_\lambda(\beta) = \mathcal{S}_{\lambda'}(\beta')$.  This is an equivalence relation on $UF_\lambda(n)$.  Denote the underlying tableau set for  $d_\lambda(\pi;x)$  by  $\mathcal{D}_\lambda(\pi)$.  We show that the polynomial coincidences actually hold at the deeper level of the underlying tableau sets:  Whenever  $s_\lambda(\beta;x) = d_\lambda(\pi;x)$,  it must be the case that  $s_\lambda(\beta;x) \equiv d_\lambda(\pi;x)$  (or  $\mathcal{S}_\lambda(\beta) = \mathcal{D}_\lambda(\pi))$.  It is not hard to see that Demazure polynomials indexed by distinct  $R$-permutations cannot be accidentally equal.  Since every flag Schur polynomial arises as a Demazure polynomial, we can conclude that two flag Schur polynomials cannot be accidentally equal.

To facilitate the development of the clearest picture of the coincidences,  when we began this project we specified unique representatives for the equivalence classes of row bounds  $\beta$  indexing the tableau sets  $\mathcal{S}_\lambda(\beta)$.  This specification process led us to reconsider the set of row bounds being used:  The weakly increasing requirement on  $\beta$  seemed needlessly restrictive,  especially when  $\lambda$  is not strict.  We decided to initially require only  $\beta_i \geq i$,  to exclude the zero polynomial.  Let  $U_\lambda(n)$  denote the set of all such  $n$-tuples.  For any  $\beta \in U_\lambda(n)$,  we define the ``row bound sum''  $s_\lambda(\beta;x)$  in the same manner as the flag Schur polynomial  $s_\lambda(\beta;x$).  We extend $\approx_\lambda$ to $U_\lambda(n)$.  For a given row bound sum generating function, the most efficient $n$-tuple $\beta$ of row bounds is increasing on any subinterval of $[n]$ for which the row lengths in $\lambda$ are constant.  If such a locally increasing  $n$-tuple is equivalent to some upper flag,  we call it a ``gapless  $\lambda$-tuple''.  Such  $n$-tuples are the most efficient row bound indexes to use when the flag Schur polynomials are being viewed more generally as row bound sums.  Let  $UGC_\lambda(n)$  be the subset of  $U_\lambda(n)$  consisting of all  $n$-tuples that are equivalent to some upper flag.  This is the largest set of row bound indexes for viewing flag Schur polynomials as row bound sums.  When  $\beta \in UGC_\lambda(n)$,  we call the row bound sum  $s_\lambda(\beta;x)$  a ``gapless core Schur polynomial''.  We honor these row bound sums with the name `Schur' since they inherit from flag Schur polynomials the properties of arising as Demazure polynomials and (therefore) not suffering from accidental equalities.  In contrast,  as is noted in Problem \ref{prob14.5},  we do not know how to rule out accidental equalities among general row bound sums.  In Sections \ref{623} and \ref{608} we describe the equivalence classes of  $\approx_\lambda$  in each of  $U_\lambda(n)  \supseteq  UGC_\lambda(n)  \supseteq  UF_\lambda(n)$  and present two systems of unique representatives in each of the three cases.

When studying the ``key'' of shape  $\lambda$  for an  $R$-permutation in \cite{PW3},  we introduced the  rank  $\lambda$-tuple map  $\Psi_\lambda$  and its inverse  $\Pi_\lambda$.  When studying the equivalence  $\approx_\lambda$  here,  we introduce the  $\lambda$-core map  $\Delta_\lambda$  on  $U_\lambda(n)$  and the notion of the ``critical list'' of  $\beta \in U_\lambda(n)$.  When $\lambda$ is not strict, the calculations involving $\beta$ usually do not need to refer to all of its entries.  The critical list of $\beta$ contains only its essential entries.  Using the maps  $\Psi_\lambda,  \Pi_\lambda$,  and $\Delta_\lambda$  and our precise indexing schemes,  we state our main results in Theorems \ref{theorem721} and \ref{theorem737.2}.  These completely describe the possible coincidences between the sets  $\mathcal{S}_\lambda(\beta)$  and  $\mathcal{D}_\lambda(\pi)$  and between the polynomials  $s_\lambda(\beta;x)$  and  $d_\lambda(\pi;x)$.  The equality   $\Delta_\lambda(\beta) = \gamma = \Psi_\lambda(\pi)$  is the central requirement for a coincidence between a gapless core Schur polynomial  $s_\lambda(\beta;x)$  and an  $R$-312-avoiding Demazure polynomial  $d_\lambda(\pi;x)$:  Here  $\beta \in UGC_\lambda(n)$,  $\gamma$  is a gapless  $\lambda$-tuple,  and  $\pi$  is an  $R$-312-avoiding  $R$-permutation.  Our foremost tools for proving the main results here were the two main results of the predecessor \cite{PW3} to this paper:  Being indexed by an  $R$-312-avoiding  $R$-permutation  $\pi$  is necessary and sufficient for a Demazure tableau set  $\mathcal{D}_\lambda(\pi)$  to be convex in  $\mathbb{Z}^N$.  These form Theorem \ref{theorem520} below.  Those two results were in turn made possible by our earlier development of tractable descriptions of all of the sets  $\mathcal{D}_\lambda(\pi)$  in \cite{Wi1} and \cite{PW1} using the scanning method of \cite{Wi1} for finding the ``right key'' of a tableau of shape  $\lambda$.

The referee for this paper remarked that our tableau convexity Theorem \ref{theorem520} is roughly paralleled by an aspect of Theorem 1.2 of \cite{KST} for Gelfand-Zetlin patterns, and that that result gives another perspective for considering coincidences between flag Schur polynomials and Demazure polynomials.  See Section 12 below.

What is the relationship between the notions of  312-avoidance for permutations and of  $R$-312-avoidance for  $R$-permutations?  Surprisingly,  the lifting and projecting answers to this question in Propositions \ref{prop824.2} and \ref{prop824.4} are expressed in terms of the maps  $\Psi_\lambda,  \Pi_\lambda$,  and  $\Delta_\lambda$ that were developed for other reasons.

Each of the two main themes of this series of papers is at least as interesting to us as is any one of the stated results in and of itself.  One of the main themes is that many of the new structures and statements arising that are parameterized by  $n$  and  $\lambda$  or  $R$  are equinumerous with the  $R$-312-avoiding  $R$-permutations.  It is reasonable to view this common count as being a ``parabolic'' generalization of the Catalan number,  since the notion of  312-avoidance is being generalized from  $n$-permutations to representatives for the cosets in the parabolic quotient  $S_n/W_J$.  (As noted at the end of \cite{PW3},  M\"{u}hle and Williams have recently independently proposed another parabolic generalization of the numbers  $C_n$.)  The other main theme of this series of papers is the ubiquity of some of the structures that are counted by the parabolic Catalan numbers.  The gapless  $\lambda$-tuples arose as the images of the  $R$-312-avoiding  $R$-permutations under the  $R$-ranking map in \cite{PW3} and they arise here as the minimum members of the equivalence classes for the indexing  $n$-tuples of the flag Schur polynomials.  Moreover,  the  $\lambda$-gapless condition provides half of the solution to the nonpermutability problem considered in our third paper \cite{PW4}.  Since the gapless  $\lambda$-tuples and the structures equivalent to them are enumerated by a parabolic generalization of Catalan numbers,  it would not be surprising if they were to arise in further contexts.

Proposition \ref{prop737}(\ref{prop737.1}) says that if  $s_\lambda(\beta;x) = s_{\lambda'}(\beta´;x)$  are two row bound sum polynomials,  then  $\lambda = \lambda'$.  Corollary \ref{newcor737}(\ref{newcor737.1}) says that one cannot then also have that $s_\lambda(\beta;x)$ and $s_{\lambda'}(\beta';x)$ are not identical as generating functions,  unless perhaps  $\beta,\beta' \in U_\lambda(n) \backslash UGC_\lambda(n)$.  If such a possibility can be ruled out by proving the non-existence of the counterexample sought by Problem \ref{prob14.5},  then Problem \ref{prob16.1} seeks an explanation for the following coincidence:  It would then be known that the number of row bound sum polynomials that cannot equal a Demazure polynomial is equal to the number of Demazure polynomials that cannot equal a row bound sum polynomial.  Further remarks on this problem and on the distinctness of the row bound sums appear in Section 12 and in Section 16 of \cite{PW2}.  In Section \ref{824B} we also relate our results on polynomial coincidences in Section \ref{737} to the results in \cite{RS} and \cite{PS};  the lifting and projecting results of Section \ref{824} are used here.  Section \ref{824B} closes with a preview of the third paper \cite{PW4} in this series.  It is surprising that gapless  $\lambda$-tuples are used there to characterize the applicability of the Gessel-Viennot method for computing row bound sums with determinants.

For this introduction,  we fixed a partition  $\lambda$  and used its shape to form a set  $R \subseteq [n-1]$.  In such a context this set should have been more precisely denoted  $R_\lambda$.  However,  many of our structures and results depend only upon some subset  $R \subseteq [n-1]$  and not upon any  $\lambda$:  The content of Sections \ref{300}, \ref{604}, \ref{608}, and \ref{824} and of parts of Sections \ref{320} and \ref{1800} are purely  $R$-theoretic.  Elsewhere  $R$  depends upon the fixed partition  $\lambda$;  then (as is noted near the end of Section \ref{800}) to avoid clutter we index the  $R_\lambda$-dependent quantities with `$\lambda$' rather than with `$R_\lambda$'.

The material in this paper first appeared as two-thirds of the manuscript \cite{PW2};  the other one-third of \cite{PW2} now constitutes \cite{PW3}.

Most definitions appear in Sections \ref{200}, \ref{300}, and \ref{800}.  The  $R$-core map   $\Delta_R$  and the critical list of  $\beta \in U_R(n)$  are studied in Section \ref{604}.  In Section \ref{320} the results needed from the predecessor paper \cite{PW3} are quoted.  Section \ref{623} sets the stage for our main results in Sections \ref{721} and \ref{737}:  The row bound tableau sets  $\mathcal{S}_\lambda(\beta)$  are introduced,  and using these the equivalence relation  $\approx_\lambda$  is defined on  $U_\lambda(n)$.  Given  $\beta \in U_\lambda(n)$,  a maximization process in \cite{RS} produced a tableau that we denote $Q_\lambda(\beta)$.  We introduce another maximization process in Section \ref{800} that produces a tableau denoted  $M_\lambda(\beta)$.  Proposition \ref{prop623.8}(\ref{prop623.8.2}) says that  $M_\lambda[\Delta_\lambda(\beta)] = Q_\lambda(\beta)$.  For a gapless  $\lambda$-tuple  $\gamma$,  Theorem 7.4 (from \cite{PW3}) says that  $M_\lambda(\gamma)$  is the key of shape  $\lambda$  of an  $R$-312-avoiding  $R$-permutation  $\pi$.  These two results provide the foundation for the bridge from the row bound tableau sets  $\mathcal{S}_\lambda(\beta)$  for  $\beta \in UGC_\lambda(n)$ to the  $R$-312-avoiding Demazure tableau sets  $\mathcal{D}_\lambda(\pi)$.

\section{General definitions}\label{200}

In posets we use interval notation to denote principal ideals and convex sets.  For example, in $\mathbb{Z}$ one has $(i, k] = \{i+1, i+2, ... , k\}$.  Given an element $x$ of a poset $P$, we denote the principal ideal $\{ y \in P : y \leq x \}$ by $[x]$.  When $P = \{1 < 2 < 3 < ... \}$, we write $[1,k]$ as $[k]$.  If $Q$ is a set of integers with $q$ elements, for $d \in [q]$ let $rank^d(Q)$ be the $d^{th}$ largest element of $Q$.  We write $\max(Q) := rank^1(Q)$ and $\min(Q) := rank^q(Q)$.  A set $\mathcal{D} \subseteq \mathbb{Z}^N$ for some $N \geq 1$ is a \emph{convex polytope} if it is the solution set for a finite system of linear inequalities.

Fix $n \geq 1$ throughout the paper.  Except for $\zeta$, various lower case Greek letters indicate various kinds of $n$-tuples of non-negative integers.  Their entries are denoted with the same letter.  An $nn$-\emph{tuple} $\nu$ consists of $n$ \emph{entries} $\nu_i \in [n]$ that are indexed by \emph{indices} $i \in [1,n]$, which together form $n$ \emph{pairs} $(i, \nu_i)$.  The 9-tuples in Table 3.1 are $nn$-tuples for $n = 9$.  Let $P(n)$ denote the poset of $nn$-tuples ordered by entrywise comparison.  It is a distributive lattice with meet and join given by entrywise min and max.  Fix an $nn$-tuple $\nu$.  A \emph{subsequence} of $\nu$ is a sequence of the form $(\nu_i, \nu_{i+1}, ... , \nu_j)$ for some $i, j \in [n]$.  The \emph{support} of this subsequence of $\nu$ is the interval $[i,j]$.  The \emph{cohort} of this subsequence of $\nu$ is the multiset $\{ \nu_k : k \in [i,j] \}$.  A \emph{staircase of $\nu$ within a subinterval $[i,j]$} for some $i, j \in [n]$ is a maximal subsequence of $(\nu_i, \nu_{i+1}, ... , \nu_j)$ whose entries increase by 1.  A \emph{plateau} in $\nu$ is a maximal constant nonempty subsequence of $\nu$;  it is \emph{trivial} if it has length 1.

An $nn$-tuple $\phi$ is a \emph{flag} if $\phi_1 \leq \ldots \leq \phi_n$.  The set of flags is a sublattice of $P(n)$;  it is essentially the lattice denoted $L(n,n)$ by Stanley.  An \emph{upper tuple} is an $nn$-tuple $\upsilon$ such that $\upsilon_i \geq i$ for $i \in [n]$.  The upper flags are the sequences of the $y$-coordinates for the above-diagonal Catalan lattice paths from $(0, 0)$ to $(n, n)$.  A \emph{permutation} is an $nn$-tuple that has distinct entries.  Let $S_n$ denote the set of permutations.  A permutation $\pi$ is $312$-\emph{avoiding} if there do not exist indices $1 \leq a < b < c \leq n$ such that $\pi_a > \pi_b < \pi_c$ and $\pi_a > \pi_c$.  Let $S_n^{312}$ denote the set of 312-avoiding permutations.  By Exercises 6.19(h) and 6.19(ff) of \cite{St2}, these permutations and the upper flags are counted by the Catalan number $C_n := \frac{1}{n+1}{2n \choose n}$.

Tableau and shape definitions are in Section \ref{800}; polynomials definitions are in Section \ref{737}.

\section{Carrels, cohorts, $\mathbf{\emph{R}}$-tuples, maps of $\mathbf{\emph{R}}$-tuples}\label{300}

Fix $R \subseteq [n-1]$ through the end of Section \ref{608}.  Denote the elements of $R$ by $q_1 < \ldots < q_r$ for some $r \geq 0$.  Set $q_0 := 0$ and $q_{r+1} := n$.  We use the $q_h$ for $h \in [r+1]$ to specify the locations of $r+1$ ``dividers'' within $nn$-tuples:  Let $\nu$ be an $nn$-tuple.  On the graph of $\nu$ in the first quadrant draw vertical lines at $x = q_h + \epsilon$ for $h \in [r+1]$ and some small $\epsilon > 0$.  These $r+1$ lines indicate the right ends of the $r+1$ \emph{carrels} $(q_{h-1}, q_h]$ \emph{of $\nu$} for $h \in [r+1]$.  An \emph{$R$-tuple} is an $nn$-tuple that has been equipped with these $r+1$ dividers.  Fix an $R$-tuple $\nu$; as in Table 3.1 we portray it by $(\nu_1, ... , \nu_{q_1} ; \nu_{q_1+1}, ... , \nu_{q_2}; ... ; \nu_{q_r+1}, ... , \nu_n)$.  Let $U_R(n)$ denote the sublattice of $P(n)$ consisting of upper $R$-tuples.  Let $UF_R(n)$ denote the sublattice of $U_R(n)$ consisting of upper flags.  Fix $h \in [r+1]$.  The $h^{th}$ carrel has $p_h := q_h - q_{h-1}$ indices.  The $h^{th}$ \emph{cohort} of $\nu$ is the multiset of entries of $\nu$ on the $h^{th}$ carrel.

An \emph{$R$-increasing tuple} is an $R$-tuple $\alpha$ that is increasing on each carrel.  Let $UI_R(n)$ denote the sublattice of $U_R(n)$ consisting of $R$-increasing upper tuples.  Consult Table 3.1 for an example and a counterexample.  Boldface entries indicate failures.  It can be seen that $|UI_R(n)| = \prod_{h=1}^{r+1} {{n-q_{h-1}} \choose {p_h}} = n! / \prod_{h=1}^{r+1} p_h! =: {n \choose {p_1 \hspace{.2pc} \dots \hspace{.2pc} p_{r+1}}} =: {n \choose R}$.  An $R$-\emph{permutation} is a permutation that is $R$-increasing when viewed as an $R$-tuple.  Let $S_n^R$ denote the set of $R$-permutations.  Note that $| S_n^R| = {n \choose R}$.  We refer to the cases $R = \emptyset$ and $R = [n-1]$ as the \emph{trivial} and \emph{full cases} respectively.  Here $| S_n^\emptyset | = 1$ and $| S_n^{[n-1]} | = n!$ respectively.  Given a permutation $\sigma \in S_n$, its \emph{$R$-projection} $\bar{\sigma} \in S^R_n$ is the $R$-increasing tuple obtained by sorting its entries in each cohort into increasing order within their carrel.  An $R$-permutation $\pi$ is $R$-$312$-\emph{containing} if there exists $h \in [r-1]$ and indices $1 \leq a \leq q_h < b \leq q_{h+1} < c \leq n$ such that $\pi_a > \pi_b < \pi_c$ and $\pi_a > \pi_c$.  An $R$-permutation is $R$-$312$-\emph{avoiding} if it is not $R$-$312$-containing.  (This is equivalent to the corresponding multipermutation being 231-avoiding.)  Let $S_n^{R\text{-}312}$ denote the set of $R$-312-avoiding permutations.  We define the \emph{$R$-parabolic Catalan number} $C_n^R$ by $C_n^R := |S_n^{R\text{-}312}|$.

\begin{figure}
\centering
\begin{tabular}{lccc}
\underline{Type of $R$-tuple} & \underline{Set} & \underline{Example} & \underline{Counterexample} \\ \\
Upper $R$-increasing tuple & $\alpha \in UI_R(n)$ & $(2,6,7;4,5,7,8,9;9)$ & $(3,5,\textbf{5};6,\textbf{4},7,8,9;9)$ \\ \\
$R$-312-avoiding permutation & $\pi \in S_n^{R\text{-}312}$  & $(2,3,6;1,4,5,8,9;7)$ & $(2,4,\textbf{6};1,\textbf{3},7,8,9;\textbf{5})$ \\ \\
Gapless core $R$-tuple & $\eta \in UGC_R(n)$ & $(4,5,5; 4,8,7,8,8;9)$ & $(4,5,5;\textbf{4},8,7,8,\textbf{9};9)$ \\ \\
Gapless $R$-tuple & $\gamma \in UG_R(n)$ & $(2,4,6;4,5,6,7,9;9)$ & $(2,4,6;\textbf{4},\textbf{6},7,8,9;9)$ \\ \\
$R$-floor flag & $\tau \in UFlr_R(n)$  & $(2,4,5;5,5,6,8,9;9)$ & $(2,4,5;5,5,\textbf{8},\textbf{8},9;9)$ \\ \\
$R$-ceiling flag & $\xi \in UCeil_R(n)$ & $(1,4,4;5,5,9,9,9;9)$ & $(1,4,4;5,5,\textbf{7},\textbf{8},9;9)$ \\ \\
\end{tabular}

\caption*{Table 3.1.  (Counter-)Examples of R-tuples for $n = 9$ and $R = \{3,8\}$.}
\end{figure}

An $R$-\emph{chain} $B$ is a sequence of sets $\emptyset =: B_0 \subset B_1 \subset \ldots \subset B_r \subset B_{r+1} := [n]$ such that $|B_h| = q_h$ for $h \in [r]$.  A bijection from $R$-permutations $\pi$ to $R$-chains $B$ is given by $B_h := \{\pi_1, \pi_2, \ldots, \pi_{q_h}\}$ for $h \in [r]$.  We indicate it by $\pi \mapsto B$.  Fix an $R$-permutation $\pi$ and let $B$ be the corresponding $R$-chain.  For $h \in [r+1]$, the set $B_h$ is the union of the first $h$ cohorts of $\pi$.  Note that $R$-chains $B$ (and hence $R$-permutations $\pi$) are equivalent to the ${n \choose R}$ objects that could be called ``ordered $R$-partitions of $[n]$''; these arise as the sequences $(B_1 \backslash B_0, B_2\backslash B_1, \ldots, B_{r+1}\backslash B_r)$ of $r+1$ disjoint nonempty subsets of sizes $p_1, p_2, \ldots, p_{r+1}$.  Now create an $R$-tuple $\Psi_R(\pi) =: \psi$ as follows:  For $h \in [r+1]$ specify the entries in its $h^{th}$ carrel by $\psi_i := \text{rank}^{q_h-i+1}(B_h)$ for $i \in (q_{h-1},q_h]$.  For a model, imagine there are $n$ discus throwers grouped into $r+1$ heats of $p_h$ throwers for $h \in [r+1]$.  Each thrower gets one throw, the throw distances are elements of $[n]$, and there are no ties.  After the $h^{th}$ heat has been completed, the $p_h$ longest throws overall so far are announced in ascending order.  See Table 3.2.  We call $\psi$ the \emph{rank $R$-tuple of $\pi$}.  As well as being $R$-increasing, it can be seen that $\psi$ is upper:  So $\psi \in UI_R(n)$.

\begin{figure}[h!]
\centering
\begin{tabular}{lccc}
\underline{Name} & \underline{From/To} & \underline{Input} & \underline{Image} \\ \\
Rank $R$-tuple & $\Psi_R:  S_n^R \rightarrow  UI_R(n)$ & $(2,4,6;1,5,7,8,9;3)$   & $(2,4,6;5,6,7,8,9;9)$ \\ \\
$R$-core & $\Delta_R: U_R(n) \rightarrow UI_R(n)$ & $(7,9,6;5,5,9,8,9;9)$ & $(4,5,6;4,5,7,8,9;9)$ \\ \\
Undoes $\Psi_R|_{S_n^{R\text{-}312}}$ & $\Pi_R:  UG_R(n) \rightarrow S_n^{R\text{-}312}$ & $(2,4,6;4,5,6,7,9;9)$  & $(2,4,6;1,3,5,7,9;8)$ \\ \\
$R$-floor & $\Phi_R: UG_R(n) \rightarrow UFlr_R(n)$ & $(3,4,6;4,5,6,8,9;9)$  & $(3,4,6;6,6,6,8,9;9)$ \\ \\
$R$-ceiling & $\Xi_R: UG_R(n) \rightarrow UCeil_R(n)$ & $(3,4,5;4,5,6,8,9;9)$  & $(5,5,5;6,6,6,9,9;9)$ \\ \\
\end{tabular}\caption*{Table 3.2.  Examples for maps of $R$-tuples for $n = 9$ and $R = \{3, 8 \}$.}
\end{figure}

We distill the crucial information from an upper $R$-tuple into a skeletal substructure called its ``critical list'', and at the same time define the \emph{$R$-core} map $\Delta_R: U_R(n) \rightarrow UI_R(n)$.  Start to refer to Figure 3.1.  For motivation read the sentences following Proposition \ref{prop623.8}.  Fix $\upsilon \in U_R(n)$.  To launch a running example, take $n := 9, R := \{3, 8 \}$, and $\upsilon := (2,7,5;8,6,6,9,9;9)$.  We will specify the image $\delta := \Delta_R(\upsilon)$.  Fix $h \in [r+1]$.  Working within the $h^{th}$ carrel $(q_{h-1}, q_h]$ from the right we recursively find for $u = 1, 2, ...$ :  At $u = 1$ the \emph{rightmost critical pair of $\upsilon$ in the $h^{th}$ carrel} is $(q_h, \upsilon_{q_h})$.  Set $x_1 := q_h$.  Recursively attempt to increase $u$ by 1:  If it exists, the \emph{next critical pair to the left} is $(x_u, \upsilon_{x_u})$, where $q_{h-1} < x_u < x_{u-1}$ is maximal such that $\upsilon_{x_{u-1}} - \upsilon_{x_u} > x_{u-1} - x_u$.  For $x_u < i \leq x_{u-1}$, write $x_{u-1} =: x$ and set $\delta_i := \upsilon_x - (x-i)$.  Otherwise, let $f_h \geq 1$ be the last value of $u$ attained.  For $q_{h-1} < i \leq x_{f_h}$, write $x_{f_h} =: x$ and again set $\delta_i := \upsilon_x - (x-i)$.  The \emph{set of critical pairs of $\upsilon$ for the $h^{th}$ carrel} is $\{ (x_u, \upsilon_{x_u}) : u \in [f_h] \} =: \mathcal{C}_h$.  Equivalently, here $f_h$ is maximal such that there exists indices $x_1, x_2, ... , x_{f_h}$ such that $q_{h-1} < x_{f_h} < ... < x_1 = q_h$ and $\upsilon_{x_{u-1}} - \upsilon_{x_u} > x_{u-1} - x_u$ for $u \in (1, f_h]$.  If $\upsilon \in UI_R(n)$, the $h^{th}$ carrel subsequence of $\upsilon$ is a concatenation of staircases in which the largest entries are the critical entries $\upsilon_{x_u}$.  The \emph{$R$-critical list for $\upsilon$} is the sequence $(\mathcal{C}_1, ... , \mathcal{C}_{r+1}) =: \mathcal{C}$ of its $r+1$ sets of critical pairs.  In our example $\mathcal{C} = ( \{ (1,2), (3,5) \}; \{(6,6),(8,9)\}; \{(9,9)\})$ and $\delta = (2,4,5;4,5,6,8,9;9)$.

\begin{figure}
  \begin{center}
  \includegraphics[scale=.85]{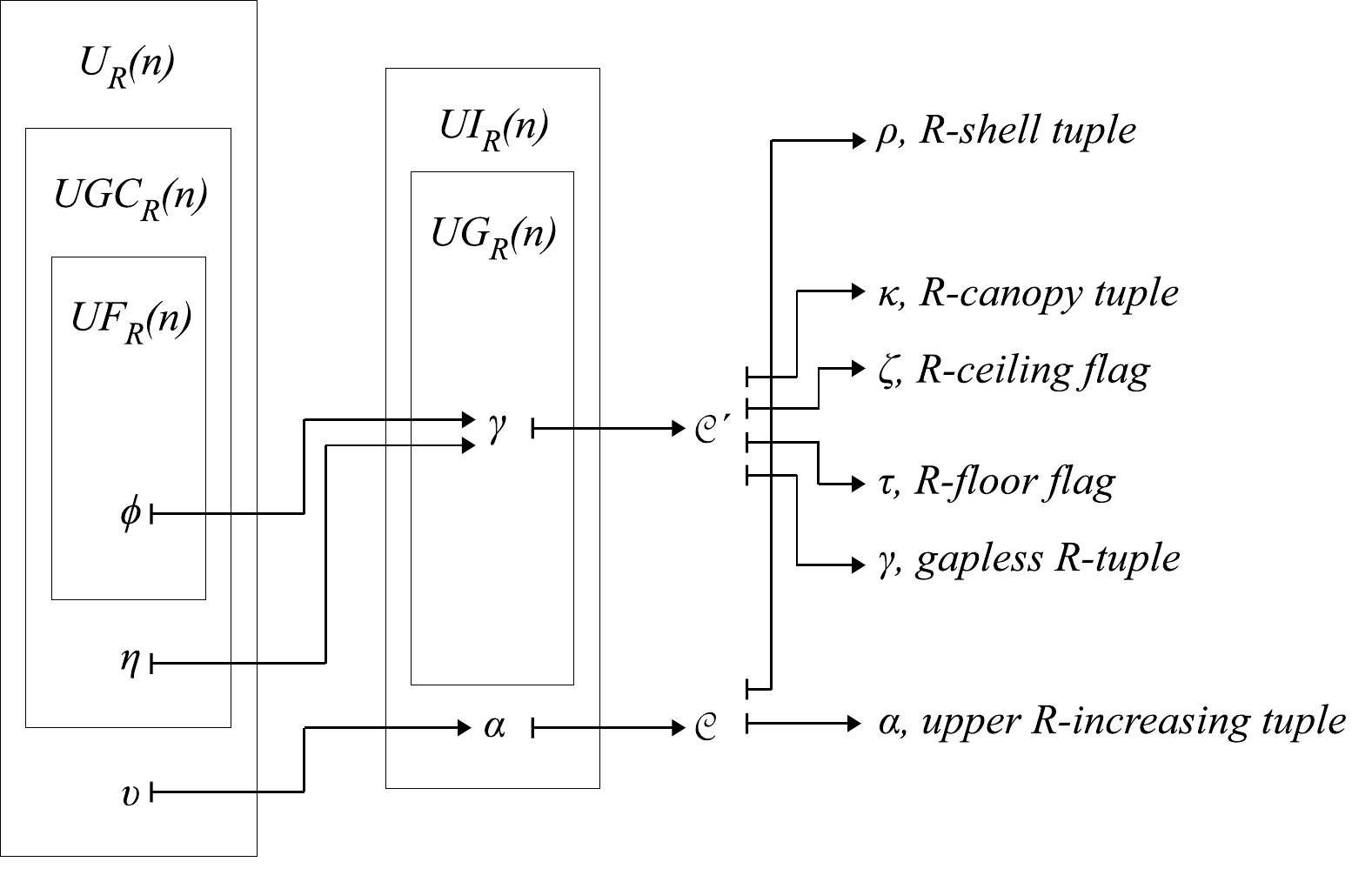}
  \caption*{Figure 3.1.  Apply the $R$-core map $\Delta_R$ and read off the $R$-critical list $\mathcal{C}$ (or the flag $R$-critical list $\mathcal{C}'$).  As in Corollary 4.4, form the two (four) standard forms that serve as class representatives in Corollary 5.3.}
  \end{center}
\end{figure}

Without having any $\upsilon$ specified, for $h \in [r+1]$ we define a set $\{ (x_u, y_{x_u}) : u \in [f_h] \} =: \mathcal{C}_h$ of pairs for some $f_h \in [p_h]$ to be a \emph{set of critical pairs} for the $h^{th}$ carrel if:  $x_u, y_{x_u} \in [n]$, $x_u \leq y_{x_u}, q_{h-1} < x_{f_h} < ... < x_1 = q_h$, and $y_{x_{u-1}} - y_{x_u} > x_{u-1} - x_u$ for $u \in (1, f_h]$.  A sequence of $r+1$ sets of critical pairs for all of the carrels is an \emph{$R$-critical list}.  The $R$-critical list of a given $\upsilon \in U_R(n)$ is an $R$-critical list.  If $(x, y_x)$ is a critical pair, we call $x$ a \emph{critical index} and $y_x$ a \emph{critical entry}.

We say that an $R$-critical list is a \emph{flag $R$-critical list} if whenever $h \in [r]$ we have $y_{q_h} \leq y_k$, where $k := x_{f_{h+1}}$.  This condition can be restated as requiring that the sequence of all of its critical entries be weakly increasing.  A \emph{gapless core $R$-tuple} is an upper $R$-tuple $\eta$ whose $R$-critical list is a flag $R$-critical list; the name comes from Proposition \ref{prop604.4}(\ref{prop604.4.3}) and the original definition of ``gapless'' presented below.  Let $UGC_R(n)$ denote the set of gapless core $R$-tuples.  The example $\upsilon$ is a gapless core $R$-tuple, since its critical list is a flag critical list.  Clearly $UF_R(n) \subseteq UGC_R(n)$.

The first part of the following definition repeats an earlier definition for the sake of symmetry.  These six kinds of $nn$-tuples will be seen in Proposition \ref{prop604.6} to arise from extending (flag) $R$-critical lists in various unique ways:

\begin{defn}Let $R \subseteq [n-1]$.
\begin{enumerate}[(i)]\setlength\itemsep{-.5em}

\item We said that $\alpha \in U_R(n)$ is an $R$-increasing upper tuple if it is increasing on each carrel.

\item We say that $\rho \in U_R(n)$ is an \emph{$R$-shell tuple} if $\rho_i = n$ for every non-critical index $i$ of $\rho$.

\item We say that $\gamma \in UGC_R(n)$ is a \emph{gapless} $R$-tuple if it is an $R$-increasing tuple.

\item We say that $\kappa \in UGC_R(n)$ is an \emph{$R$-canopy tuple} if it is an $R$-shell tuple.

\item We say that $\tau \in UF_R(n)$ is an \emph{$R$-floor flag} if the leftmost pair of each non-trivial plateau in $\tau$ has the form $(q_h, \tau_{q_h})$ for some $h \in [r]$.

\item We say that $\xi \in UF_R(n)$ is an \emph{$R$-ceiling flag} if it is a concatenation of plateaus whose rightmost pairs are the $R$-critical pairs of $\xi$.  \end{enumerate} \end{defn}

\noindent The example $\delta$ above is a gapless $R$-tuple.  Let $UG_R(n)$, $UFlr_R(n)$, and $UCeil_R(n)$ respectively denote the sets of gapless $R$-tuples, $R$-floor flags, and of $R$-ceiling flags.

In \cite{PW3} we defined a gapless $R$-tuple to be an $R$-increasing upper tuple $\gamma$ such that whenever there exists $h \in [r]$ with $\gamma_{q_h} > \gamma_{q_h+1}$, then $\gamma_{q_h} - \gamma_{q_h+1} + 1 =: s \leq p_{h+1}$ and the first $s$ entries of the $(h+1)^{st}$ carrel $(q_h, q_{h+1} ]$ are $\gamma_{q_h}-s+1, \gamma_{q_h}-s+2, ... , \gamma_{q_h}$.  This definition is equivalent to requiring for all $h \in [r]$:  If $\gamma_{q_h} > \gamma_{q_{h}+1}$, then the leftmost staircase within the $(h+1)^{st}$ carrel must contain an entry $\gamma_{q_h}$.  In Proposition \ref{prop604.4}(\ref{prop604.4.3}) we show that that definition is equivalent to the present one.  In Proposition 4.6(ii) of \cite{PW3} we showed that the restriction of $\Psi_R$ to $S_n^{R\text{-}312}$ is a bijection to $UG_R(n)$ whose inverse is the following map $\Pi_R$:  Let $\gamma \in UG_R(n)$.  See Table 3.2.  Define an $R$-tuple $\Pi_R(\gamma) =: \pi$ by:  Initialize $\pi_i := \gamma_i$ for $i \in (0,q_1]$.  Let $h \in [r]$.  If $\gamma_{q_h} > \gamma_{q_h+1}$, set $s:= \gamma_{q_h} - \gamma_{q_h+1} + 1$.  Otherwise set $s := 0$.  For $i$ in the right side $(q_h + s, q_{h+1}]$ of the $(h+1)^{st}$ carrel, set $\pi_i := \gamma_i$.  For $i$ in the left side $(q_h, q_h + s]$, set $d := q_h + s - i + 1$ and $\pi_i := rank^d( \hspace{1mm} [\gamma_{q_h}] \hspace{1mm} \backslash \hspace{1mm} \{ \pi_1, ... , \pi_{q_h} \} \hspace{1mm} )$.  In words:  working from right to left, fill in the left side by first finding the largest element of $[\gamma_{q_h}]$ not used by $\pi$ so far, then the next largest, and so on.  In Table 3.2 when $h=1$ the elements 5, 3, 1 are found and placed into the 6$^{th}$, 5$^{th}$, and 4$^{th}$ positions.  (Since $\gamma$ is a gapless $R$-tuple, when $s \geq 1$ we have $\gamma_{q_h + s} = \gamma_{q_h}$.  Since `gapless' includes the upper property, here we have $\gamma_{q_h +s} \geq q_h + s$.  Hence $| \hspace{1mm} [\gamma_{q_h}] \hspace{1mm} \backslash \hspace{1mm} \{ \pi_1, ... , \pi_{q_h} \} \hspace{1mm} | \geq s$, and so there are enough elements available to define these left side $\pi_i$. )  Since $\gamma_{q_h} \leq \gamma_{q_{h+1}}$, it can inductively be seen that $\max\{ \pi_1, ... , \pi_{q_h} \} = \gamma_{q_h}$.

When we restrict our attention to the full $R = [n-1]$ case in Sections 11 and 12, we suppress most prefixes and subscripts of `$R$' (and of `$\lambda$') from maps, sets of $n$-tuples, and terminologies.  Clearly $UGC(n) = UG(n) = UF(n) = UFlr(n) = UCeil(n)$.  The  number of $nn$-tuples in each of these sets is $C_n$.

\section{Cores, shells, gapless tuples, canopies, floors, ceilings}\label{604}

In this section we use the critical list substructure to relate six kinds of $R$-tuples that can be used as indexes for the row bound tableau sets in Section \ref{623}.  Over Section \ref{300}, this section, and Section \ref{608} we are defining three versions of some of these notions, which have a word such as `floor' in common in their names.  When delineation of these three interrelated concepts is needed, one should consult the summary paragraph at the end of Section \ref{608}.

The process used to define the $R$-core map can also be used to bijectively produce the $R$-tuples in $UI_R(n)$ from the set of all $R$-critical lists:  To see surjectivity, note that the staircases within the carrels of a given $\alpha \in UI_R(n)$ can be formed toward the left from the critical pairs of $\alpha$.

\begin{fact}\label{fact604.2}Let $\upsilon \in U_R(n)$ and set $\Delta_R(\upsilon) =: \delta \in UI_R(n)$.  Here $\delta \leq \upsilon$ in $U_R(n)$ and $\delta$ has the same critical list as $\upsilon$.  Another upper $R$-tuple $\upsilon^\prime$ has the same critical list as $\upsilon$ if and only if $\Delta_R(\upsilon^\prime) = \Delta_R(\upsilon)$.  If $\upsilon \in UI_R(n)$, then $\Delta_R(\upsilon) = \upsilon$.  \end{fact}

Part (ii) of the following statement can be used to characterize the upper $R$-tuples with flag $R$-critical lists by referring to the original definition of gapless $R$-tuple.

\begin{prop}\label{prop604.4}Let $\upsilon \in U_R(n), \eta \in UGC_R(n),$ and $\phi \in UF_R(n)$.
\begin{enumerate}[(i)]\setlength\itemsep{-.5em}

\item The $R$-cores $\Delta_R(\eta)$ and $\Delta_R(\phi)$ of $\eta$ and $\phi$ are gapless $R$-tuples. \label{prop604.4.1}

\item One has $\Delta_R(\upsilon) \in UG_R(n)$ if and only if $\upsilon \in UGC_R(n)$. \label{prop604.4.2}

\item The two definitions of gapless $R$-tuple are equivalent. \label{prop604.4.3} \end{enumerate} \end{prop}

\begin{proof}For (i), note that these images are in $UI_R(n)$ and have the same $R$-critical lists as did $\eta$ and $\phi$.  One direction of (ii) restates part of (i), and the other direction follows from $\Delta_R$ preserving the $R$-critical list.  Let $\gamma \in UG_R(n) \subseteq UI_R(n)$.  For $h \in [r]$, the entry $\gamma_{q_h}$ does not exceed the leftmost critical entry on the $(h+1)^{st}$ carrel.  Since that entry is the maximum entry in the leftmost staircase for $\gamma$ on this carrel, if $\gamma_{q_h} > \gamma_{q_h+1}$ then an entry in that staircase must be $\gamma_{q_h}$.  Conversely, suppose that $\gamma \in UI_R(n)$ satisfies the original definition to be $R$-gapless.  The visualization above can be reversed for all $h \in [r]$ to see that $\gamma_{q_h}$ will not exceed the leftmost critical entry on the $(h+1)^{st}$ carrel, whether $\gamma_{q_h} > \gamma_{q_h+1}$ or not.  \end{proof}

Most of our kinds of $R$-tuples correspond bijectively to $R$-critical lists or to flag $R$-critical lists.  The following six $R$-tuples $\alpha, \rho, \gamma, \kappa, \tau,$ and $\xi$ will be considered in the proposition below.  Let $\mathcal{C}$ be an $R$-critical list.  For each critical pair $(x, y_x)$ in $\mathcal{C}$, if $x$ is the leftmost critical index set $x^\prime := 0$;  otherwise let $x^\prime$ be the largest critical index that is less than $x$.  Set $\xi_x := \tau_x := \kappa_x := \gamma_x := \rho_x := \alpha_x := y_x$.  Then for $x^\prime < i < x$:  Set $\alpha_i := \alpha_x - (x-i)$.  Set $\rho_i := n$.  Now suppose that $\mathcal{C}$ is a flag $R$-critical list.  Set $\gamma_i := \gamma_x - (x-i)$.  Set $\kappa_i := n$.  If $x$ is the leftmost critical index in the $(h+1)^{st}$ carrel for some $h \in [r]$, then $x^\prime = q_h$ and we set $\tau_i := \max \{ \tau_{q_h}, \tau_x - (x-i) \}$ for $i \in (q_h, x)$.  Otherwise set $\tau_i := \tau_x - (x-i)$ for $i \in (x^\prime, x)$.  Set $\xi_i := \xi_x$.

\begin{prop}\label{prop604.6}Let $\mathcal{C}$ be an $R$-critical list.

\begin{enumerate}[(i)]\setlength\itemsep{-.5em}

\item The $R$-tuples $\alpha$ and $\rho$ above are respectively the unique $R$-increasing upper tuple and the unique $R$-shell tuple whose $R$-critical lists are $\mathcal{C}$. \label{prop604.6.1}

\item If $\mathcal{C}$ is a flag $R$-critical list, the $R$-tuples $\gamma, \kappa, \tau,$ and $\xi$ above are respectively the unique gapless $R$-tuple, the unique $R$-canopy tuple, the unique $R$-floor flag, and the unique $R$-ceiling flag whose $R$-critical lists are $\mathcal{C}$. \label{prop604.6.2} \end{enumerate} \end{prop}

\begin{proof}It is clear that the $R$-critical list of each of these six tuples is the given $R$-critical list.  We confirm that the six definitions are satisfied:  Since the $R$-tuples $\alpha$ and $\gamma$ are produced as in the definition of the $R$-core map $\Delta_R$, we see that $\alpha, \gamma \in UI_R(n)$.  Since the $R$-critical list given for $\gamma$ is a flag $R$-critical list, we have $\gamma \in UG_R(n)$.  Clearly $\rho$ is an $R$-shell tuple.  Since $\kappa = \rho$, the flag $R$-critical list hypothesis implies that $\kappa$ is an $R$-canopy tuple.  If $\tau$ has a non-trivial plateau it must occur when $\tau_i$ is set to $\tau_{q_h}$ for some $h \in [r]$ and some consecutive indices $i$ at the beginning of the $(h+1)^{st}$ carrel.  If this $\tau_{q_h}$ is greater than $\tau_{q_h-1}$ then the definition of $R$-floor flag is satisfied.  Otherwise $\tau_{q_h} = \tau_{q_h-1}$, which implies that all entries in the $h^{th}$ carrel have the value $\tau_{q_{h-1}}$.  This plateau will necessarily terminate at the rightmost entry in some earlier carrel, since the entries in the first carrel are strictly increasing.  Clearly $\xi$ is an $R$-ceiling flag.

For the uniqueness of $\alpha$:  It was noted earlier that this construction is bijective from $R$-critical lists to $UI_R(n)$; its injective inverse is the formation of the $R$-critical list.  Restrict this bijection to the flag $R$-critical lists to get uniqueness for $\gamma$.  It is clear from the definitions of $R$-shell tuple, $R$-canopy tuple, and $R$-ceiling flag that for each of these notions any two $R$-tuples with the same $R$-critical list must also have the same non-critical entries.  Let $\tau^\prime$ be any $R$-floor flag with flag $R$-critical list $\mathcal{C}$.  Let $h \in [r]$.  Let $x$ be the leftmost critical index in $(q_h, q_{h+1}]$.  Then for $i \in (q_h, x)$ it can be seen that the critical entries at $q_h$ and $x$ force $\tau_i^\prime = \max \{ \tau_{q_h}^\prime, \tau_x^\prime - (x-i) \}$.  On $(x, q_{h+1}]$ and $(0, q_1]$ the flag $\tau^\prime$ must be increasing.  So on $(x, q_{h+1})$ and $(0, q_1)$ the entries of $\tau^\prime$ are uniquely determined by the critical pairs via staircase decomposition for $\alpha$.  \end{proof}

\noindent Here we say that $\alpha$ and $\rho$ are respectively the \emph{$R$-increasing upper tuple} and the \emph{$R$-shell tuple} \emph{for the $R$-critical list $\mathcal{C}$}.  We also say that $\gamma, \kappa, \tau,$ and $\xi$ are respectively the \emph{gapless $R$-tuple}, the \emph{$R$-canopy tuple}, the \emph{$R$-floor flag}, and the \emph{$R$-ceiling flag} \emph{for the flag $R$-critical list $\mathcal{C}$}.

\begin{cor}\label{cor604.8}The six constructions above specify bijections from the set of $R$-critical lists (flag $R$-critical lists) to the sets of $R$-increasing upper tuples and $R$-shell tuples (gapless $R$-tuples, $R$-canopy tuples, $R$-floor flags, and $R$-ceiling flags).  \end{cor}

\begin{proof}These maps are injective since they preserve the (flag) $R$-critical lists.  To show surjectivity, first find the $R$-critical list of the target $R$-tuple. \end{proof}

\section{Equivalence classes in $\mathbf{\emph{U}}_\mathbf{\emph{R}}\mathbf{\emph{(n)}}$,
$\mathbf{\emph{UGC}}_\mathbf{\emph{R}}\mathbf{\emph{(n)}}$,
$\mathbf{\emph{UF}}_\mathbf{\emph{R}}\mathbf{\emph{(n)}}$}\label{608}

Here we present results needed to study the sets of tableaux of shape $\lambda$ with given row bounds in Section \ref{623}.  There we reduce that study to the study of the following sets of $R$-increasing tuples, after we determine $R := R_\lambda \subseteq [n-1]$ from $\lambda$:  For $\beta \in U_R(n)$, set $\{ \beta \}_R := \{ \epsilon \in UI_R(n):  \epsilon \leq \beta \}$.  This is not \'{a} priori a principal ideal in $UI_R(n)$, since it is possible that $\beta \notin UI_R(n)$.  But we will see that for any $\beta$ there exists $\alpha \in UI_R(n)$ such that $\{ \beta \}_R$ is the principal ideal $[\alpha]$ in $UI_R(n)$.

Define an equivalence relation $\sim_R$ on $U_R(n)$ as follows:  Let $\upsilon, \upsilon^\prime \in U_R(n)$.  We define $\upsilon \sim_R \upsilon^\prime$ if $\{ \upsilon \}_R = \{ \upsilon^\prime \}_R$.  Sometimes we restrict $\sim_R$ from $U_R(n)$ to $UGC_R(n)$, or further to $UF_R(n)$.  We denote the equivalence classes of $\sim_R$ in these three sets respectively by $\langle \upsilon \rangle_{\sim_R}$, $\langle \eta \rangle_{\sim_R}^G$, and $\langle \phi \rangle_{\sim_R}^F$.  We indicate intervals in $UGC_R(n)$ and $UF_R(n)$ respectively with $[ \cdot , \cdot ]^G$ and $[ \cdot, \cdot ]^F$.

The proofs for the results stated in this section appear in \cite{PW2}.

\begin{lem}\label{lemma608.2}Let $\upsilon \in U_R(n), \eta \in UGC_R(n)$, and $\phi \in UF_R(n)$.
\begin{enumerate}[(i)]\setlength\itemsep{-.5em}

\item Here $\{ \upsilon \}_R = [\Delta_R(\upsilon)] \subseteq UI_R(n)$.  So $\upsilon' \sim_R \upsilon$ for some $\upsilon' \in U_R(n)$ if and only if $\Delta_R(\upsilon') = \Delta_R(\upsilon)$ if and only if $\upsilon'$ has the same $R$-critical list as $\upsilon$. \label{lemma608.2.1}

\item The equivalence classes $\langle \upsilon \rangle_{\sim_R}$, $\langle \eta \rangle_{\sim_R}^G$, and $\langle \phi \rangle_{\sim_R}^F$ are closed respectively in $U_R(n), UGC_R(n)$, and $UF_R(n)$ under the meet and the join operations for $U_R(n)$. \label{lemma608.2.2} \end{enumerate} \end{lem}

\noindent So by Part (i) we can view these three equivalence classes as consisting of $R$-tuples that share (flag) $R$-critical lists.  And by Part (ii), each of these equivalence classes has a unique minimal and a unique maximal element under the entrywise partial orders.  We denote the minimums of $\langle \upsilon \rangle_{\sim_R} \subseteq U_R(n)$ and of $\langle \eta \rangle_{\sim_R}^G \subseteq UGC_R(n)$ by $\utilde{\upsilon}$ and \d{$\eta$} respectively.  We call the maximums of $\langle \upsilon \rangle_{\sim_R} \subseteq U_R(n)$ and of $\langle \eta \rangle_{\sim_R}^G \subseteq UGC_R(n)$ the \emph{$R$-shell of $\upsilon$} and the \emph{$R$-canopy of $\eta$} and denote them by $\tilde{\upsilon}$ and $\dot{\eta}$ respectively.  For the class $\langle \phi \rangle_{\sim_R}^F \subseteq UF_R(n)$, we call and denote these respectively the \emph{$R$-floor \b{$\phi$} of $\phi$} and the \emph{$R$-ceiling $\bar{\phi}$ of $\phi$}.

These definitions give the containments $\langle \upsilon \rangle_{\sim_R} \subseteq [$$\utilde{\upsilon}$,$\tilde{\upsilon}], \langle \eta \rangle_{\sim_R}^G \subseteq [$\d{$\eta$}$,\dot{\eta}]^G$, and $\langle \phi \rangle_{\sim_R}^F \subseteq [$\b{$\phi$},$\bar{\phi}]^F$ for our next result:

\begin{prop}\label{prop608.4}Let $\upsilon \in U_R(n), \eta \in UGC_R(n)$, and $\phi \in UF_R(n)$.
\begin{enumerate}[(i)]\setlength\itemsep{-.5em}

\item Here $\utilde{\upsilon}$ $ = \Delta_R(\upsilon)$, the $R$-core of $\upsilon$.  In $U_R(n)$ we have $\langle \upsilon \rangle_{\sim_R} = [$$\utilde{\upsilon}$, $\tilde{\upsilon}]$.  The $R$-core $\utilde{\upsilon}$ of $\upsilon$ (respectively $R$-shell $\tilde{\upsilon}$ of $\upsilon$) is the $R$-increasing upper tuple (respectively $R$-shell tuple) for the $R$-critical list of $\upsilon$. \label{prop608.4.1}

\item We have $[$$\utilde{\upsilon}$, $\tilde{\upsilon}] \subseteq UGC_R(n)$ or $[$$\utilde{\upsilon}$, $\tilde{\upsilon}] \subseteq U_R(n) \backslash UGC_R(n)$, depending on whether $\upsilon \in UGC_R(n)$ or not.  We also have \d{$\eta$} = $\utilde{\eta}$ and $\dot{\eta} = \tilde{\eta}$.  And $\langle \eta \rangle_{\sim_R}^G =  [$\d{$\eta$}, $\dot{\eta}]^G = [\utilde{\eta}, \tilde{\eta}] = \langle \eta \rangle_{\sim_R}$:  The equivalence classes $UGC_R(n) \supseteq \langle \eta \rangle_{\sim_R}^G$ and $\langle \eta \rangle_{\sim_R} \subseteq U_R(n)$ are the same subset of $U_R(n)$, which is an interval in both contexts.  The $R$-core \d{$\eta$} of $\eta$ (respectively $R$-canopy $\dot{\eta}$ of $\eta$) is the gapless $R$-tuple (respectively $R$-canopy tuple) for the flag $R$-critical list of $\eta$. \label{prop608.4.2}

\item In $UF_R(n)$ we have $\langle \phi \rangle_{\sim_R}^F =  [$\b{$\phi$}, $\bar{\phi}]^F$.  The $R$-floor \b{$\phi$} of $\phi$ (respectively $R$-ceiling $\bar{\phi}$ of $\phi$) is the $R$-floor flag (respectively $R$-ceiling flag) for the flag $R$-critical list of $\phi$.  We have $[$\b{$\phi$}, $\bar{\phi}]^F \subseteq $ \\ $[$\d{$\phi$}, $\dot{\phi} ] = \langle \phi \rangle_{\sim_R} \subseteq UGC_R(n)$. \label{prop608.4.3} \end{enumerate} \end{prop}

\noindent So for $\upsilon \in U_R(n)$ the equivalence classes $\langle \upsilon \rangle_{\sim_R}$ are intervals $[ \utilde{\upsilon}, \tilde{\upsilon}]$ that lie entirely in $U_R(n) \backslash UGC_R(n)$ or entirely in $UGC_R(n)$, in which case they coincide with the equivalence classes $\langle \eta \rangle_{\sim_R}^G = [$\d{$\eta$}, $\dot{\eta}]^G$ for $\eta \in UGC_R(n)$ originally defined by restricting $\sim_R$ to $UGC_R(n)$.  However, although for $\phi \in UF_R(n)$ the equivalence class $\langle \phi \rangle_{\sim_R}^F$ is an interval $[$\b{$\phi$}, $\bar{\phi}]^F$ when working within $UF_R(n)$, it can be viewed as consisting of some of the elements of the interval $[$\d{$\phi$}, $\dot{\phi} ]^G$ of $UGC_R(n)$ (or of $U_R(n)$) that is formed by viewing $\phi$ as an element of $UGC_R(n)$.

\begin{cor}\label{cor608.6}The equivalence classes of $\sim_R$ can be indexed as follows:
\begin{enumerate}[(i)]\setlength\itemsep{-.5em}

\item In $U_R(n)$, they are precisely indexed by the $R$-increasing upper tuples or the $R$-shell tuples (or by the $R$-critical lists). \label{cor608.6.1}

\item In $UGC_R(n)$, they are precisely indexed by the gapless $R$-tuples or the $R$-canopy tuples (or by the flag $R$-critical lists, the $R$-floor flags, or the $R$-ceiling flags). \label{cor608.6.2}

\item In $UF_R(n)$, they are precisely indexed by the $R$-floor flags or the $R$-ceiling flags (or by the flag $R$-critical lists, the gapless $R$-tuples, or the $R$-canopy tuples). \label{cor608.6.3} \end{enumerate} \end{cor}

If the gapless $R$-tuple label for an equivalence class in $UF_R(n)$ is not a flag, we may want to convert it to the unique $R$-floor (or $R$-ceiling) flag that belongs to the same class.  Let $\gamma \in UG_R(n)$ and find the flag $R$-critical list of $\gamma$.  As in Section \ref{604}, compute the $R$-floor flag $\tau$ and the $R$-ceiling $\xi$ for this flag $R$-critical list.  Define the \emph{$R$-floor map} $\Phi_R : UG_R(n) \longrightarrow UFlr(n)$ and \emph{$R$-ceiling map} $\Xi_R: UG_R(n) \longrightarrow UCeil_R(n)$ by $\Phi_R(\gamma) := \tau$ and $\Xi_R(\gamma) := \xi$.  See Figure 5.1.

\begin{prop}\label{prop608.10}The maps $\Phi_R:  UG_R(n) \longrightarrow UFlr_R(n)$ and $\Xi_R:  UG_R(n) \longrightarrow UCeil_R(n)$ are bijections; each has inverse $\Delta_R$.  \end{prop}

To summarize:  In Section \ref{300} the six notions of $R$-increasing upper tuple, $R$-shell tuple, gapless $R$-tuple, $R$-canopy tuple, $R$-floor flag, and $R$-ceiling flag were defined with conditions on the entries of an $R$-tuple.  While introducing the word `for' into these terms, in Section \ref{604} one such $R$-tuple was associated to each (flag) $R$-critical list.  While introducing the word `of' into four of these terms, in this section these kinds of $R$-tuples arose as the extreme elements of equivalence classes.  This began with the classes in $U_R(n)$.  Here these extreme elements were respectively $R$-increasing upper and $R$-shell tuples.  When these classes were restricted to the subset $UGC_R(n)$ of upper $R$-tuples with gapless cores, these extreme elements were respectively gapless $R$-tuples and $R$-canopy tuples.  When these classes were restricted further to the subset $UF_R(n)$ of upper flags, these extreme elements were respectively $R$-floor and $R$-ceiling flags.

\begin{figure}[h!]
  \begin{center}
  \includegraphics{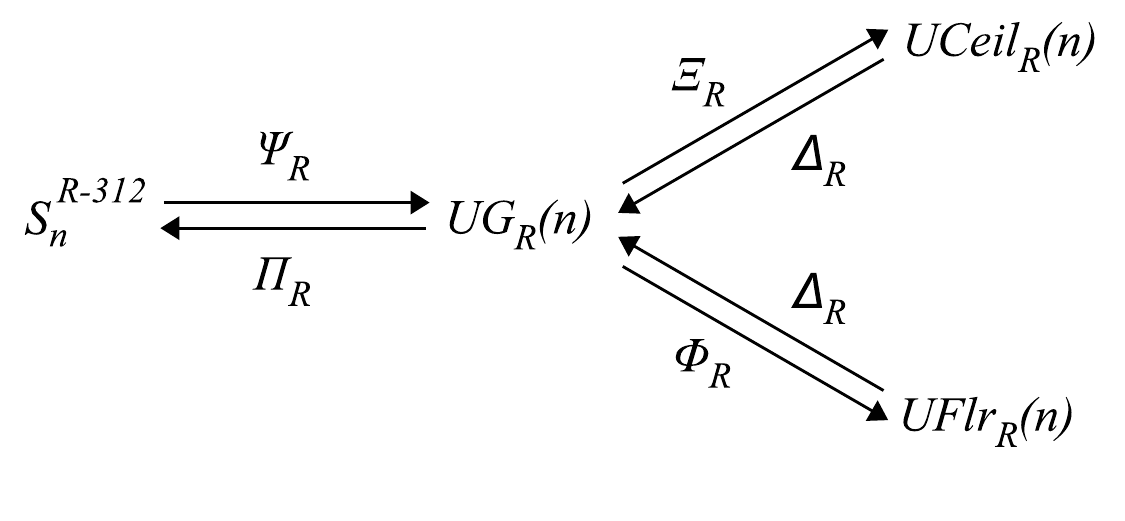}
  \caption*{Figure 5.1.  Proposition 5.4 and Proposition 7.2(ii)}
  \end{center}
\end{figure}

\section{Shapes, tableaux, connections to Lie theory}\label{800}

A \emph{partition} is an $n$-tuple $\lambda \in \mathbb{Z}^n$ such that $\lambda_1 \geq \ldots \geq \lambda_n \geq 0$.  Fix such a $\lambda$ for the rest of the paper.  We say it is \emph{strict} if $\lambda_1 > \ldots > \lambda_n$.  The \emph{shape} of $\lambda$, also denoted $\lambda$, consists of $n$ left justified rows with $\lambda_1, \ldots, \lambda_n$ boxes.  We denote its column lengths by $\zeta_1 \geq \ldots \geq \zeta_{\lambda_1}$.  The column length $n$ is called the \emph{trivial} column length.  Since the columns are more important than the rows, the boxes of $\lambda$ are transpose-indexed by pairs $(j,i)$ such that $1 \leq j \leq \lambda_1$ and $1 \leq i \leq \zeta_j$.  Sometimes for boundary purposes we refer to a $0^{th}$ \emph{latent column} of boxes, which is a prepended $0^{th}$ column of trivial length.  If $\lambda = 0$, its shape is the \emph{empty shape} $\emptyset$.  Define $R_\lambda \subseteq [n-1]$ to be the set of distinct non-trivial column lengths of $\lambda$.  Note that $\lambda$ is strict if and only if $R_\lambda = [n-1]$, i.e. $R_\lambda$ is full.  Set $|\lambda| := \lambda_1 + \ldots + \lambda_n$.

A \emph{(semistandard) tableau of shape $\lambda$} is a filling of $\lambda$ with values from $[n]$ that strictly increase from north to south and weakly increase from west to east.  Let $\mathcal{T}_\lambda$ denote the set of tableaux $T$ of shape $\lambda$.  Under entrywise comparison $\leq$, this set $\mathcal{T}_\lambda$ becomes a poset that is the distributive lattice $L(\lambda, n)$ introduced by Stanley.  The principal ideals $[T]$ in $\mathcal{T}_\lambda$ are clearly convex polytopes in $\mathbb{Z}^{|\lambda|}$.  Fix $T \in \mathcal{T}_\lambda$.  For $j \in [\lambda_1]$, we denote the one column ``subtableau'' on the boxes in the $j^{th}$ column by $T_j$.  Here for $i \in [\zeta_j]$ the tableau value in the $i^{th}$ row is denoted $T_j(i)$.  The set of values in $T_j$ is denoted $B(T_j)$.  Columns $T_j$ of trivial length must be \emph{inert}, that is $B(T_j) = [n]$.  The $0^{th}$ \emph{latent column} $T_0$ is an inert column that is sometimes implicitly prepended to the tableau $T$ at hand:  We ask readers to refer to its values as needed to fulfill definitions or to finish constructions.  We say a tableau $Y$ of shape $\lambda$ is a $\lambda$-\emph{key} if $B(Y_l) \supseteq B(Y_j)$ for $1 \leq l \leq j \leq \lambda_1$.  To define the \emph{content $\Theta(T) := \theta$ of $T$}, for $i \in [n]$ take $\theta_i$ to be the number of values in $T$ equal to $i$.  The empty shape has one tableau on it, the \emph{null tableau}.  Fix a set $Q \subseteq [n]$ with $|Q| =: q \geq 0$.  The \emph{column} $Y(Q)$ is the tableau on the shape for the partition $(1^q, 0^{n-q})$ whose values form the set $Q$.  Then for $d \in [q]$, the value in the $(q+1-d)^{th}$ row of $Y(Q)$ is $rank^d(Q)$.

The most important values in a tableau of shape $\lambda$ occur at the ends of its rows.  Using the latent column when needed, these $n$ values from $[n]$ are gathered into an $R_\lambda$-tuple as follows:  Let $T \in \mathcal{T}_\lambda$.  The \emph{$\lambda$-row end list} $\omega$ of $T$ is the $R_\lambda$-tuple defined by $\omega_i := T_{\lambda_i}(i)$ for $i \in [n]$.  Note that for $h \in [r+1]$ one has $\lambda_i = \lambda_{i^\prime}$ for $i, i^\prime \in (q_{h-1}, q_{h} ]$.  For $h \in [r+1]$ the entries in the $h^{th}$ cohort of $\omega$ are increasing.  So $\omega \in UI_{R_\lambda}(n)$.

For $h \in [r]$, the columns of length $q_h$ in the shape $\lambda$ have indices $j$ such that $j \in (\lambda_{q_{h+1}}, \lambda_{q_h}]$.  A bijection from $R$-chains $B$ to $\lambda$-keys $Y$ is obtained by juxtaposing from left to right $\lambda_n$ inert columns and $\lambda_{q_h}-\lambda_{q_{h+1}}$ copies of $Y(B_h)$ for $r \geq h \geq 1$.   We indicate it by $B \mapsto Y$.  A bijection from $R_\lambda$-permutations $\pi$ to $\lambda$-keys $Y$ is obtained by following $\pi \mapsto B$ with $B \mapsto Y$.  The image of an $R_\lambda$-permutation $\pi$ is called the \emph{$\lambda$-key of $\pi$}; it is denoted $Y_\lambda(\pi)$.  It is easy to see that the $\lambda$-row end list of the $\lambda$-key of $\pi$ is the rank $R_\lambda$-tuple $\Psi_{R_\lambda}(\pi) =: \psi$ of $\pi$:  Here $\psi_i = Y_{\lambda_i}(i)$ for $i \in [n]$.

Let $\alpha \in UI_{R_\lambda}(n)$.  Define $\mathcal{Z}_\lambda(\alpha)$ to be the subset of tableaux $T \in \mathcal{T}_\lambda$ that have $\lambda$-row end list $\alpha$.  To see that $\mathcal{Z}_\lambda(\alpha) \neq \emptyset$, for $i \in [n]$ take $T_j(i) := i$ for $j \in [1, \lambda_i)$ and $T_{\lambda_i}(i) := \alpha_i$.  This subset is closed under the join operation for the lattice $\mathcal{T}_\lambda$.  We define the \emph{$\lambda$-row end max tableau $M_\lambda(\alpha)$ for $\alpha$} to be the unique maximal element of $\mathcal{Z}_\lambda(\alpha)$.  The definition of $Q_\lambda(\beta)$, a close relative to $M_\lambda(\alpha)$, can be found in Section \ref{623}.

When we are considering tableaux of shape $\lambda$, much of the data used will be in the form of $R_\lambda$-tuples.  Many of the notions used will be definitions from Section \ref{300} that are being applied with $R := R_\lambda$.  The structure of each proof will depend only upon $R_\lambda$ and not upon any other aspect of $\lambda$:  If $\lambda^\prime$ and $\lambda^{\prime\prime}$ are partitions such that $R_{\lambda^\prime} = R_{\lambda^{\prime\prime}}$, then the development for $\lambda^{\prime\prime}$ will in essence be the same as for $\lambda^\prime$.  To emphasize the original independent entity $\lambda$ and to reduce clutter, from now on rather than writing `$R$' or `$R_\lambda$' we will replace `$R$' by `$\lambda$' in subscripts and in prefixes.  Above we would have written $\omega \in UI_\lambda(n)$ instead of having written $\omega \in UI_{R_\lambda}(n)$ (and instead of having written $\omega \in UI_R(n)$ after setting $R := R_\lambda$).  When $\lambda$ is a strict partition, we omit the `$\lambda$-' prefixes and the subscripts since $R_\lambda = [n-1]$.

To connect to Lie theory, fix $R \subseteq [n-1]$ and set $J := [n-1] \backslash R$.  The $R$-permutations are the one-rowed forms of the ``inverses'' of the minimum length representatives collected in $W^J$ for the cosets in $W /W_J$, where $W$ is the Weyl group of type $A_{n-1}$ and $W_J$ is its parabolic subgroup $\langle s_i: i \in J \rangle$.  A partition $\lambda$ is strict exactly when the weight it depicts for $GL(n)$ is strongly dominant.  If we take the set $R$ to be $R_\lambda$, then the restriction of the partial order $\leq$ on $\mathcal{T}_\lambda$ to the $\lambda$-keys depicts the Bruhat order on that $W^J$.  Consult the second and third paragraphs of Section \ref{737} for the Demazure and flag Schur polynomials.  Further details appear in Sections 2, 3,  and the appendix of \cite{PW1}.

\section{Results needed from the first paper}\label{320}

The content of this section should be referred to when it is cited in later sections.  We begin with the material that will be needed from Section 4 of \cite{PW3}:

Given a set of integers, a \emph{clump} of it is a maximal subset of consecutive integers.  After decomposing a set into its clumps, we index the clumps in the increasing order of their elements.  For example, the set $\{ 2,3,5,6,7,10,13,14 \}$ is the union $L_1 \hspace{.5mm} \cup  \hspace{.5mm} L_2  \hspace{.5mm} \cup \hspace{.5mm}  L_3  \hspace{.5mm} \cup  \hspace{.5mm} L_4$,  where  $L_1 := \{ 2,3 \},  L_2 := \{ 5,6,7 \}, L_3 := \{ 10 \},  L_4 := \{ 13,14 \}$.

We return to our fixed $R \subseteq [n-1]$.  (See Section 4 of \cite{PW3} for the simple full case $R = [n-1]$.)  Let $B$ be an $R$-chain.  We say $B$ is \emph{$R$-rightmost clump deleting} if this condition holds for each $h \in [r]$:  Let $B_{h+1} =: L_1 \cup L_2 \cup ... \cup L_f$ decompose $B_{h+1}$ into clumps for some $f \geq 1$.  We require $L_e \cup L_{e+1} \cup ... \cup L_f \supseteq B_{h+1} \backslash B_{h} \supseteq L_{e+1} \cup ... \cup L_f$ for some $e \in [f]$.  This condition requires the set $B_{h+1} \backslash B_h$ of new elements that augment the set $B_h$ of old elements to consist of entirely new clumps $L_{e+1}, L_{e+2}, ... , L_f$, plus some further new elements that combine with some old elements to form the next ``lower'' clump $L_e$ in $B_{h+1}$.  Here are some reformulations of the notion of $R$-rightmost clump deleting:

\begin{fact}\label{fact320.3}Let $B$ be an $R$-chain.  For $h \in [r]$, set $b_{h+1} := \min (B_{h+1} \backslash B_{h} )$ and $m_h := \max (B_h)$.  This $R$-chain is $R$-rightmost clump deleting if and only if each of the following holds:
\begin{enumerate}[(i)]\setlength\itemsep{-.5em}

\item  For $h \in [r]$, one has $[b_{h+1}, m_{h}] \subseteq B_{h+1}$. \label{fact320.3.1}

\item For $h \in [r]$, one has $(b_{h+1}, m_{h}) \subset B_{h+1}$. \label{fact320.3.2}

\item For $h \in [r]$, let $s$ be the number of elements of $B_{h+1} \backslash B_{h}$ that are less than $m_{h}$.  These must be the $s$ largest elements of $[m_{h}] \backslash B_{h}$. \label{fact320.3.3} \end{enumerate} \end{fact}

In the proofs of both parts of the following result, one takes $(b_{h+1}, m_h)$ in Part (ii) above to be $( \min\{\pi_{q_{h}+1}, ... , \pi_{q_{h+1}} \} , \max \{\pi_1, ... , \pi_{q_{h}}\} )$ for $\pi \in S_n^R$ under the correspondence $\pi \leftrightarrow B$.  The bijection in Part (ii) below generalizes the inverse bijection for Theorem 14.1 of \cite{PS}.

\begin{prop}\label{prop320.2}For $R \subseteq [n-1]$ we have:
\begin{enumerate}[(i)]\setlength\itemsep{-.5em}

\item The restriction of the global bijection $\pi \mapsto B$ from $S_n^R$ to $S_n^{R\text{-}312}$ is a bijection to the set of $R$-rightmost clump deleting chains. \label{prop320.2.1}

\item The restriction of the rank $R$-tuple map $\Psi_R$ from $S_n^R$ to $S_n^{R\text{-}312}$ is a bijection to $UG_R(n)$ whose inverse is $\Pi_R$.  \label{prop320.2.2} \end{enumerate} \end{prop}

We also need some material from Section 5 of \cite{PW3}.  We return to our fixed partition $\lambda$.

\begin{lem}\label{lemma340.1}The $\lambda$-row end max tableau $M_{\lambda}(\gamma)$ of a gapless $\lambda$-tuple $\gamma$ is a key.  \end{lem}

A $\lambda$-key $Y$ is \emph{gapless} if the condition below is satisfied for $h \in [r-1]$:  Let $b$ be the smallest value in a column of length $q_{h+1}$ that does not appear in a column of length $q_{h}$.  For $j \in (\lambda_{q_{h +2}}, \lambda_{q_{h+1}}]$, let $i \in (0, q_{h+1}]$ be the shared row index for the occurrences of $b = Y_j(i)$.  Let $m$ be the bottom (largest) value in the columns of length $q_{h}$.  If $b > m$ there are no requirements.  Otherwise:  For $j \in (\lambda_{q_{h +2}}, \lambda_{q_{h+1}}]$, let $k \in (i, q_{h+1}]$ be the shared row index for the occurrences of $m = Y_j(k)$.  For $j \in (\lambda_{q_{h + 2}}, \lambda_{q_{h+1}}]$ one must have $Y_j(i+1) = b+1, Y_j(i+2) = b+2, ... , Y_j(k-1) = m-1$ holding between $Y_j(i) = b$ and $Y_j(k) = m$.  (Hence necessarily $m - b = k - i$.)

\begin{thm}\label{theorem340}Let $\lambda$ be a partition and set $R := R_\lambda$.
\begin{enumerate}[(i)]\setlength\itemsep{-.5em}

\item An $R$-permutation $\pi$ is $R$-312-avoiding if and only if its $\lambda$-key $Y_\lambda(\pi)$ is gapless. \label{theorem340.1}

\item If an $R$-permutation $\pi$ is $R$-312-avoiding, then the $\lambda$-row end max tableau $M_\lambda(\gamma)$ of its rank $R$-tuple $\Psi_R(\pi) =: \gamma$ is its $\lambda$-key $Y_\lambda(\pi)$. \label{theorem340.2} \end{enumerate} \end{thm}

\noindent In the full case, the converse of Part (ii) holds:  If the row end max tableau of the rank tuple of a permutation is the key of the permutation, then the permutation is 312-avoiding. In \cite{PW3} it is further noted that an $R$-tuple is $R$-gapless here if and only if it arises as the $\lambda$-row end list of a gapless $\lambda$-key.

We conclude with material needed from Sections 6 and 7 of \cite{PW3}.  Fix a $\lambda$-permutation $\pi$.  To completely present our definition of the set $\mathcal{D}_\lambda(\pi)$ of Demazure tableaux, we would need to first specify how to find the ``scanning tableau'' $S(T)$ for a given $T \in \mathcal{T}_\lambda$.  Since that one paragraph specification is not explicitly used in this paper, we only note that $S(T) \in \mathcal{T}_\lambda$ and refer curious readers to \cite{PW3}.  It was shown in \cite{Wi1} that $S(T)$ is the ``right key'' of Lascoux and Sch\"{u}tzenberger for $T$.  As in \cite{PW1}, we use the $\lambda$-key $Y_\lambda(\pi)$ of $\pi$ and $S(T)$ to define the set of \emph{Demazure tableaux}:  $\mathcal{D}_\lambda(\pi) := \{ T \in \mathcal{T}_\lambda : S(T) \leq Y_\lambda(\pi) \}$.  We list some basic facts concerning keys, scanning tableaux, and sets of Demazure tableaux.  Part (i) is elementary; the other parts were justified in \cite{PW3}.

\begin{fact}\label{fact420}Let $T \in \mathcal{T}_\lambda$ and let $Y \in \mathcal{T}_\lambda$ be a key. \begin{enumerate}[(i)]\setlength\itemsep{-.5em}

\item If $\Theta(Y) = \Theta(U)$ for some $U \in \mathcal{T}_\lambda$, then $U = Y$. \label{fact420.1}

\item $S(T)$ is a key and hence $S(T) \in \mathcal{T}_\lambda$. \label{fact420.2}

\item $T \leq S(T)$ and $S(Y) = Y$. \label{fact420.3}

\item $Y_\lambda(\pi) \in \mathcal{D}_\lambda(\pi)$ and $\mathcal{D}_\lambda(\pi) \subseteq [Y_\lambda(\pi)]$. \label{fact420.4}

\item The unique maximal element of $\mathcal{D}_\lambda(\pi)$ is $Y_\lambda(\pi)$. \label{fact420.5}

\item The Demazure sets $\mathcal{D}_\lambda(\sigma)$ of tableaux are nonempty subsets of $\mathcal{T}_\lambda$ that are precisely indexed by the $\sigma \in S_n^\lambda$.  \label{fact420.6} \end{enumerate}\end{fact}

The main results of \cite{PW3} were:

\begin{thm}\label{theorem520}Let $\lambda$ be a partition and let $\pi$ be a $\lambda$-permutation.  The set $\mathcal{D}_\lambda(\pi)$ of Demazure tableaux of shape $\lambda$ is a convex polytope in $\mathbb{Z}^{|\lambda|}$ if and only if $\pi$ is $\lambda$-312-avoiding if and only if $\mathcal{D}_\lambda(\pi) = [Y_\lambda(\pi)]$.  \end{thm}

\section{Sets of tableaux specified by row bounds}\label{623}

We use the $R$-tuples studied in Section \ref{604} and \ref{608} to develop precise indexing schemes for row bound tableau sets.

Determine the subset $R_\lambda \subseteq [n-1]$ for our fixed partition $\lambda$.  We must temporarily suspend our notation shortcuts regarding `$R_\lambda$'.  Let $\beta$ be an $R_\lambda$-tuple.  We define the \emph{row bound set of tableaux} to be $\mathcal{S}_\lambda(\beta) := \{ T \in \mathcal{T}_\lambda : T_j(i) \leq \beta_i \text{ for } j \in [0, \lambda_1] \text{ and } i \in [\zeta_j] \}$.  The set $\mathcal{S}_\lambda(\beta)$ is non-empty if and only if $\beta$ is upper:  This condition is clearly necessary.  For sufficiency, set $\delta := \Delta_R(\beta)$ and note that $\delta$ is upper.  Then since $\delta \in UI_{R_\lambda}(n)$ we can re-use the $T \in \mathcal{T}_\lambda$ given in Section \ref{800} to see that $\emptyset \neq \mathcal{Z}_\lambda(\delta) \subseteq \mathcal{S}_\lambda(\beta)$.  Henceforth we assume that $\beta$ is upper:  $\beta \in U_{R_\lambda}(n)$.  To interface with the literature for flagged Schur functions, we give special attention to the \emph{flag bound sets} $\mathcal{S}_\lambda(\varphi)$ for upper flags $\varphi \in UF_{R_\lambda}(n)$.  We also want to name the row bound sets $\mathcal{S}_\lambda(\eta)$ for $\eta \in UGC_{R_\lambda}(n)$; we call these the \emph{gapless core bound sets}.

We can focus on the row ends of the tableaux at hand because $\mathcal{S}_\lambda(\beta) = \{ T \in \mathcal{T}_\lambda : T_{\lambda_i}(i) \leq \beta_i \text{ for } i \in [n] \}$.  Let $\alpha \in U_{R_\lambda}(n)$.  In Section \ref{800} we noted that $\mathcal{Z}_\lambda(\alpha) \neq \emptyset$ if and only if $\alpha \in UI_{R_\lambda}(n)$.  These $\mathcal{Z}_\lambda(\alpha)$ are disjoint for distinct $\alpha \in UI_{R_\lambda}(n)$.  Clearly $\mathcal{S}_\lambda(\beta) = \bigcup \mathcal{Z}_\lambda(\alpha)$, taking the union over the $\alpha$ in the subset $\{ \beta \}_{R_\lambda} \subseteq UI_{R_\lambda}(n)$ defined in Section \ref{608}.  This observation and Lemma \ref{lemma608.2} allow us to study the three kinds of row bound sets $\mathcal{S}_\lambda(\beta)$ by considering the principal ideals $[\Delta_{R_\lambda}(\beta)] = \{ \beta \}_{R_\lambda}$ of $UI_{R_\lambda}(n)$ for $\beta \in U_{R_\lambda}(n)$ or $\beta \in UGC_{R_\lambda}(n)$ or $\beta \in UF_{R_\lambda}(n)$.  The results of Sections \ref{604} and \ref{608} can be used to show:

\begin{prop}\label{prop623.1}Let $\beta \in U_{R_\lambda}(n)$.  The row bound set $\mathcal{S}_\lambda(\beta)$ is a flag bound tableau set if and only if $\beta \in UGC_{R_\lambda}(n)$.  For the ``if'' statement use $\mathcal{S}_\lambda(\beta) = \mathcal{S}_\lambda(\varphi)$ for $\varphi := \Phi_{R_\lambda}[\Delta_{R_\lambda}(\beta)]$. \end{prop}

\noindent At times for indexing reasons we will prefer the ``gapless core'' viewpoint.

When $\lambda$ is not strict, it is possible to have $\mathcal{S}_\lambda(\beta) = \mathcal{S}_\lambda(\beta^\prime)$ for distinct $\beta, \beta^\prime \in U_{R_\lambda}(n)$.  We want to study how much such labelling ambiguity is present for $\mathcal{S}_\lambda(.)$, and we want to develop unique labelling systems for the tableau sets $\mathcal{S}_\lambda(\beta), \mathcal{S}_\lambda(\eta)$, and $\mathcal{S}_\lambda(\varphi)$.  For $\beta, \beta^\prime \in U_{R_\lambda}(n)$, define $\beta \approx_{\lambda} \beta^\prime$ if $\mathcal{S}_\lambda(\beta) = \mathcal{S}_\lambda(\beta^\prime)$.  Sometimes we restrict $\approx_\lambda$ to $UGC_{R_\lambda}(n)$ or further to $UF_{R_\lambda}(n)$.  We denote the equivalence class of $\beta \in U_{R_\lambda}(n), \eta \in UGC_{R_\lambda}(n)$, and $\varphi \in UF_{R_\lambda}(n)$ by $\langle \beta \rangle_\lambda, \langle \eta \rangle_\lambda^G$, and $\langle \varphi \rangle_\lambda^F$.  By the $\mathcal{S}_\lambda(\beta) = \bigcup \mathcal{Z}_\lambda(\alpha)$ observation and the fact that the $\mathcal{Z}_\lambda(\alpha)$ are non-empty and disjoint, it can be seen that this relation $\approx_\lambda$ on $U_{R_\lambda}(n)$ or $UGC_{R_\lambda}(n)$ or $UF_{R_\lambda}(n)$ is the same as the relation $\sim_{R_\lambda}$ defined on these sets in Section \ref{608}:  Each relation can be expressed in terms of principal ideals of $UI_{R_\lambda}(n)$.

\begin{prop}\label{prop623.2}On the sets $U_{R_\lambda}(n), UGC_{R_\lambda}(n)$, and $UF_{R_\lambda}(n)$, the relation $\approx_\lambda$ coincides with the relation $\sim_{R_\lambda}$.  \end{prop}

\noindent Now that we know that these relations coincide, we can safely return to replacing `$R_\lambda$' with `$\lambda$' in subscripts and prefixes.  Definitions and results from Sections \ref{604} and \ref{608} will be used by always taking $R := R_\lambda$.  In particular, the unique labelling systems listed in Corollary \ref{cor608.6} for the equivalence classes of $\sim_{\lambda}$ can now be used for the equivalence classes of $\approx_\lambda$.  Henceforth we more simply write `$\sim$' for `$\approx_\lambda$'.

We state some applications of the work in Sections \ref{604} and \ref{608} to the current context:

\begin{prop}\label{prop623.4}Below we take $\beta, \beta^\prime \in U_\lambda(n)$ and $\eta, \eta' \in UGC_\lambda(n)$ and $\varphi, \varphi^\prime \in UF_\lambda(n)$:
\begin{enumerate}[(i)]\setlength\itemsep{-.5em}

\item The row bound sets $\mathcal{S}_\lambda(\beta)$ are precisely indexed by the $\lambda$-increasing upper tuples $\alpha \in UI_\lambda(n)$, which are the minimal representatives in $U_\lambda(n)$ for the equivalence classes $\langle \beta \rangle_\lambda$.  One has $\mathcal{S}_\lambda(\beta) = \mathcal{S}_\lambda(\beta^\prime)$ if and only if $\beta \sim \beta^\prime$ if and only if $\Delta_\lambda(\beta) = \Delta_\lambda(\beta^\prime)$. \label{prop623.4.1}

\item The gapless core bound sets $\mathcal{S}_\lambda(\eta)$ are precisely indexed by the gapless $\lambda$-tuples $\gamma \in UG_\lambda(n)$, which are the minimal representatives in $UGC_\lambda(n)$ for the equivalence classes $\langle \eta \rangle_\lambda^G = \langle \eta \rangle_\lambda$.  One has $\mathcal{S}_\lambda(\eta) = \mathcal{S}_\lambda(\eta^\prime)$ if and only if $\eta \sim \eta^\prime$ if and only if $\Delta_\lambda(\eta) = \Delta_\lambda(\eta^\prime)$. \label{prop623.4.2}

\item The flag bound sets $\mathcal{S}_\lambda(\varphi)$ are precisely indexed by the $\lambda$-floor flags $\tau \in UFlr_\lambda(n)$, which are the minimal representatives in $UF_\lambda(n)$ for the equivalence classes $\langle \varphi \rangle_\lambda^F$.  One has $\mathcal{S}_\lambda(\varphi) = \mathcal{S}_\lambda(\varphi^\prime)$ if and only if $\varphi \sim \varphi^\prime$ if and only if  $\Phi_\lambda[\Delta_\lambda(\varphi)] = \Phi_\lambda[\Delta_\lambda(\varphi^\prime)]$.  The flag bound sets $\mathcal{S}_\lambda(\varphi)$ can also be faithfully depicted as the sets $\mathcal{S}_\lambda(\gamma)$ for $\gamma \in UG_\lambda(n)$ by taking $\gamma := \Delta_\lambda(\varphi)$. \label{prop623.4.3} \end{enumerate} \end{prop}

Let $\beta \in U_\lambda(n)$.  Following Theorem 23 of \cite{RS}, we define the \emph{$\lambda$-row bound max tableau} $Q_\lambda(\beta)$ to be the least upper bound in $\mathcal{T}_\lambda$ of the tableaux in $\mathcal{S}_\lambda(\beta)$.  It can be seen that $Q_\lambda(\beta) \in  \mathcal{S}_\lambda(\beta)$.

\begin{prop}\label{prop623.8}Let $\beta, \beta^\prime \in U_\lambda(n)$ and set $\Delta_\lambda(\beta) =: \delta \in UI_\lambda(n)$. \begin{enumerate}[(i)]\setlength\itemsep{-.5em}

\item Here $\mathcal{S}_\lambda(\beta) = [Q_\lambda(\beta)]$ and so $\mathcal{S}_\lambda(\beta) = \mathcal{S}_\lambda(\beta^\prime)$ if and only if $Q_\lambda(\beta) = Q_\lambda(\beta^\prime)$. \label{prop623.8.1}

\item Here $M_\lambda(\delta) = Q_\lambda(\beta)$ and so $\mathcal{S}_\lambda(\beta) = [M_\lambda(\delta)]$. \label{prop623.8.2} \end{enumerate} \end{prop}

\noindent So if one wants to view $\mathcal{S}_\lambda(\beta)$ as a principal ideal, the $\lambda$-row end list of its generator $Q_\lambda(\beta)$ is $\Delta_\lambda(\beta)$.  Here the critical entries of $\beta$ are exactly those that agree with the corresponding entries of $\Delta_\lambda(\beta)$; the other entries of $\beta$ have been replaced with the smallest possible entries.

\begin{proof}Only the first claim in (ii) is not already evident:  Recall that $\mathcal{S}_\lambda(\beta) = \bigcup \mathcal{Z}_\lambda(\alpha)$, union over $\alpha \in \{ \beta \}_\lambda \subseteq UI_\lambda(n)$.  By Proposition \ref{prop623.2} and Lemma \ref{lemma608.2}(\ref{lemma608.2.1}) we have $\{ \beta \}_\lambda = [\delta]$.  So $\mathcal{S}_\lambda(\beta) = \bigcup \mathcal{Z}_\lambda(\alpha) = \mathcal{S}_\lambda(\delta)$, union over $\alpha \leq \delta$ in $UI_\lambda(n)$.  Here $\delta \in UI_\lambda(n)$ implies $\mathcal{Z}_\lambda(\delta) \neq \emptyset$.  Let $U \in \mathcal{Z}_\lambda(\delta)$ and $T \in \mathcal{Z}_\lambda(\alpha)$ for some $\alpha \in UI_\lambda(n)$.  Here $\alpha < \delta$ would imply $T \ngtr U$.  Hence $Q_\lambda(\beta) \in \mathcal{Z}_\lambda(\delta)$.  So both $Q_\lambda(\beta)$ and $M_\lambda(\delta)$ are the maximum tableau of $\mathcal{Z}_\lambda(\delta)$.  \end{proof}

\section{Coincidences of row or flag bound sets with \\ Demazure tableau sets}\label{721}

When can one set of tableaux arise both as a row bound set $\mathcal{S}_\lambda(\beta)$ for some upper $\lambda$-tuple $\beta$ and as a Demazure set $\mathcal{D}_\lambda(\pi)$ for some $\lambda$-permutation $\pi$?  Since we will seek coincidences between flag Schur polynomials $s_\lambda(\varphi; x)$ and Demazure polynomials $d_\lambda(\pi; x)$, we should also pose this question for the flag bound set $\mathcal{S}_\lambda(\varphi)$ for some upper flag $\varphi$.  Our deepest result of \cite{PW3} gave a necessary condition for a Demazure tableau set to be convex.  Initially we refer to it here for guiding motivation.  Then we use it to prove the hardest part, Part (iii) for necessity, of the theorem below.  Our next deepest result of \cite{PW3} implied a sufficient condition for a Demazure tableau set to be convex.  We use it here to prove Part (ii), for sufficiency, of the theorem below.

For motivation, note that any set $\mathcal{S}_\lambda(\beta)$ is convex in $\mathbb{Z}^{|\lambda|}$: Proposition \ref{prop623.8}(\ref{prop623.8.1}) says that it is the principal ideal $[Q_\lambda(\beta)]$ of $\mathcal{T}_\lambda$, where $Q_\lambda(\beta)$ is the $\lambda$-row bound max tableau for $\beta$.  And Theorem \ref{theorem520} says that $\mathcal{D}_\lambda(\pi)$ is convex only if the $\lambda$-permutation $\pi$ is $\lambda$-312-avoiding.  So, to begin the proof of Part (ii) below, fix $\pi \in S_n^{\lambda\text{-}312}$.  Find the key $Y_\lambda(\pi)$ of $\pi$.  Theorem \ref{theorem520} says that $\mathcal{D}_\lambda(\pi) = [Y_\lambda(\pi)]$.  (Since $[Y_\lambda(\pi)]$ is convex, we now know that the sets $\mathcal{D}_\lambda(\pi)$ for such $\pi$ are exactly the convex candidates to arise in the form $\mathcal{S}_\lambda(\beta)$.)  Form the rank $\lambda$-tuple $\Psi_\lambda(\pi) =: \gamma$ of $\pi$.  By Proposition \ref{prop320.2}(\ref{prop320.2.2}) we know that $\gamma$ is a gapless $\lambda$-tuple:  $\gamma \in UG_\lambda(n)$.  Theorem \ref{theorem340}(\ref{theorem340.2}) says that $Y_\lambda(\pi) = M_\lambda(\gamma)$, the $\lambda$-row end max tableau for $\gamma$.  Proposition \ref{prop623.8}(\ref{prop623.8.2}) gives $M_\lambda(\gamma) = Q_\lambda(\gamma)$, since $UG_\lambda(n) \subseteq UI_\lambda(n)$ by definition and $\Delta_\lambda(\gamma) = \gamma$ by Fact \ref{fact604.2}.  So $\mathcal{D}_\lambda(\pi) = [Q_\lambda(\gamma)]$.  Hence by Proposition \ref{prop623.8}(\ref{prop623.8.1}) it arises as $\mathcal{S}_\lambda(\gamma) = [Q_\lambda(\gamma)]$.

Parts (i) and (ii) of the following theorem give sufficient conditions for a coincidence from two perspectives, Part (iii) gives necessary conditions for a coincidence, and Part (iv) presents a neutral precise indexing.  But the theorem statement begins with a less technical summary:

\begin{thm}\label{theorem721}Let $\lambda$ be a partition.  A row bound set $\mathcal{S}_\lambda(\beta)$ of tableaux for an upper $\lambda$-tuple $\beta$ arises as a Demazure set if and only if the $\lambda$-core $\Delta_\lambda(\beta)$ of $\beta$ is a gapless $\lambda$-tuple.  Therefore every flag bound set of tableaux arises as a Demazure set, and a row bound set arises as a Demazure set if and only if it arises as a flag bound set.  A Demazure set $\mathcal{D}_\lambda(\pi)$ of tableaux for a $\lambda$-permutation $\pi$ arises as a row bound set if and only if $\pi$ is $\lambda$-312-avoiding.  Specifically:
\begin{enumerate}[(i)]\setlength\itemsep{-.5em}

\item Let $\beta \in U_\lambda(n)$.  If $\beta \in UGC_\lambda(n)$, set $\pi := \Pi_\lambda[\Delta_\lambda(\beta)]$. Then $\mathcal{S}_\lambda(\beta) = \mathcal{D}_\lambda(\pi)$, and $\pi$ is the unique $\lambda$-permutation for which this is true. Here $\pi \in S_n^{\lambda\text{-}312}$. \label{theorem721.1}

\item Let $\pi \in S_n^\lambda$.  If $\pi \in S_n^{\lambda\text{-}312}$, set $\gamma := \Psi_\lambda(\pi)$. Then $\mathcal{D}_\lambda(\pi) = \mathcal{S}_\lambda(\gamma)$, and $\mathcal{D}_\lambda(\pi) = \mathcal{S}_\lambda(\beta)$ for some $\beta \in U_\lambda(n)$ implies $\Delta_\lambda(\beta) = \gamma$. Here $\gamma \in UG_\lambda(n)$ and so $\beta \in UGC_\lambda(n)$. \label{theorem721.2}

\item Suppose some $\beta \in U_\lambda(n)$ and some $\pi \in S_n^\lambda$ exist such that $\mathcal{S}_\lambda(\beta) = \mathcal{D}_\lambda(\pi)$. Then one has $Q_\lambda(\beta) = Y_\lambda(\pi)$ and $\Delta_\lambda(\beta) = \Psi_\lambda(\pi)$.  Here $\beta \in UGC_\lambda(n)$ and $\pi \in S_n^{\lambda\text{-}312}$. \label{theorem721.3}

\item The collection of the sets $\mathcal{S}_\lambda(\varphi)$ for $\varphi \in UF_\lambda(n)$ is the same as the collection of sets $\mathcal{D}_\lambda(\pi)$ for $\pi \in S_n^{\lambda\text{-}312}$. These collections can be simultaneously precisely indexed by $\gamma \in UG_\lambda(n)$ as follows: Given such a $\gamma$, produce $\Phi_\lambda(\gamma) =: \varphi \in UFlr_\lambda(n)$ and $\Pi_\lambda(\gamma) =: \pi \in S_n^{\lambda\text{-}312}$. \label{theorem721.4} \end{enumerate} \end{thm}

\begin{proof}First confirm (i) - (iv):  The first and last two claims in (ii) were deduced above.  For (i), use Proposition \ref{prop320.2}(\ref{prop320.2.2}) to see $\pi \in S_n^{\lambda\text{-}312}$ and to express $\Delta_\lambda(\beta)$ as $\Psi_\lambda(\pi)$.  The first claim in (ii) then tells us that $\mathcal{S}_\lambda[\Delta_\lambda(\beta)] = \mathcal{D}_\lambda(\pi)$.  But Proposition \ref{prop623.4}(\ref{prop623.4.1}) gives $\mathcal{S}_\lambda[\Delta_\lambda(\beta)] = \mathcal{S}_\lambda(\beta)$.  We return to the second claims in (i) and (ii) after we confirm (iii).  So suppose we have $\mathcal{S}_\lambda(\beta) = \mathcal{D}_\lambda(\pi)$.  Since $\mathcal{S}_\lambda(\beta)$ is a principal ideal in $\mathcal{T}_\lambda$, Theorem \ref{theorem520} tells us that we must have $\pi \in S_n^{\lambda\text{-}312}$.  The unique maximal elements of $\mathcal{S}_\lambda(\beta)$ and of $\mathcal{D}_\lambda(\pi)$ (see Proposition \ref{prop623.8}(\ref{prop623.8.1}) and Fact \ref{fact420}(\ref{fact420.5})) must coincide: $Q_\lambda(\beta) = Y_\lambda(\pi)$.  Via consideration of $M_\lambda[\Delta_\lambda(\beta)]$, Proposition \ref{prop623.8}(\ref{prop623.8.2}) implies that the row end list of $Q_\lambda(\beta)$ is $\Delta_\lambda(\beta)$.  Section \ref{800} noted that the row end list of $Y_\lambda(\pi)$ is $\Psi_\lambda(\pi)$.  Hence $\Delta_\lambda(\beta) = \Psi_\lambda(\pi) \in UG_\lambda(n)$.  The uniqueness in (i) is obtained by applying the inverse $\Pi_\lambda$ of $\Psi_\lambda$ to the requirements in (iii) that $\pi \in S_n^{\lambda\text{-}312}$ and that $\Psi_\lambda(\pi) = \Delta_\lambda(\beta)$.  The uniqueness-up-to-$\Delta_\lambda$-equivalence in (ii) restates the second claim of (iii).  For (iv): Proposition \ref{prop623.4}(\ref{prop623.4.3}) says that the collection of the sets $\mathcal{S}_\lambda(\varphi)$ is precisely indexed by the $\lambda$-floor flags.  By restricting Fact \ref{fact420}(\ref{fact420.6}), these sets $\mathcal{D}_\lambda(\pi)$ are already precisely indexed by their $\lambda$-312-avoiding permutations.  By Proposition \ref{prop608.10} and Proposition \ref{prop320.2}(\ref{prop320.2.2}), apply the bijections $\Delta_\lambda$ and $\Psi_\lambda$ to re-index these collections with gapless $\lambda$-tuples.  Use (ii) and Proposition \ref{prop623.4}(\ref{prop623.4.2}) to see that the same set $\mathcal{S}_\lambda(\varphi) = \mathcal{D}_\lambda(\pi)$ will arise from a given gapless $\lambda$-tuple $\gamma$ when these re-indexings are undone via $\varphi := \Phi_\lambda(\gamma)$ and $\pi := \Pi_\lambda(\gamma)$.  Three of the four initial summary statements of this theorem should now be apparent.  The third statement follows from Proposition \ref{prop623.1}. \end{proof}

\section{Flagged Schur functions and key polynomials}\label{737}

We use Theorem \ref{theorem721} to improve upon the results in \cite{RS} and \cite{PS} concerning coincidences between flag Schur polynomials and Demazure polynomials.

Let $x_1,\ldots, x_n$ be indeterminants.  Let $T \in \mathcal{T}_\lambda$.  The \emph{weight} $x^{\Theta(T)}$ of $T$ is $x_1^{\theta_1}\ldots x_n^{\theta_n}$, where $\theta$ is the content $\Theta(T)$.

Let $\beta$ be an upper $\lambda$-tuple:  $\beta \in U_\lambda(n)$.  We introduce the \emph{row bound sum} $s_\lambda(\beta; x) := \sum x^{\Theta(T)}$, sum over $T \in \mathcal{S}_\lambda(\beta)$.  In particular, to relate to the literature \cite{RS} \cite{PS}, at times we restrict our attention to flag row bounds.  Here for $\varphi \in UF_\lambda(n)$ we define the \emph{flag Schur polynomial} to be $s_\lambda(\varphi; x)$.  (Often `upper' is not required at the outset; if $\varphi$ is not upper then the empty sum would yield $0$ for the polynomial.  Following Stanley we write `flag' instead of `flagged' \cite{St2}, and following Postnikov and Stanley we write `polynomial' for `function' \cite{PS}.)  More generally, for $\eta \in UGC_\lambda(n)$ we define the \emph{gapless core Schur polynomial} to be $s_\lambda(\eta;x)$.

Let $\pi$ be a $\lambda$-permutation:  $\pi \in S_n^\lambda$.  Here we define the \emph{Demazure polynomial} $d_\lambda(\pi; x)$ to be $\sum x^{\Theta(T)}$, sum over $T \in \mathcal{D}_\lambda(\pi)$.  For Lie theorists, we make two remarks:  Using the Appendix and Sections 2 and 3 of \cite{PW1}, via the right key scanning method and the divided difference recursion these polynomials can be identified as the Demazure characters for $GL(n)$ and as the specializations of the key polynomials $\kappa_\alpha$ of \cite{RS} to a finite number of variables.  In the axis basis, the highest and lowest weights for the corresponding Demazure module are $\lambda$ and $\Theta[Y_\lambda(\pi)]$.  (Postnikov and Stanley chose `Demazure character' over `key polynomial' \cite{PS}.  By using `Demazure polynomial' for the $GL(n)$ case, which should be recognizable to Lie theorists, we leave `Demazure character' available for general Lie type.)

We say that two polynomials that are defined as sums of the weights over sets of tableaux are \emph{identical as generating functions} if the two tableau sets coincide.  So for row bound sums we write $s_\lambda(\beta; x) \equiv s_{\lambda'}(\beta'; x)$ if and only if $\lambda = \lambda'$ and then $\beta \sim \beta'$.  It is conceivable that the polynomial equality $s_\lambda(\beta; x) = s_{\lambda'}(\beta'; x)$ could ``accidentally'' hold between two non-identical row bound sums, that is when $\lambda \neq \lambda^\prime$ and/or $\beta \nsim \beta'$.

It is likely that Part (ii) of the following preliminary result can be deduced from the sophisticated Corollary 7 of \cite{RS}, which states that the Demazure polynomials form a basis of the polynomial ring.  One would need to show that specializing $x_{n+1} = x_{n+2} = ... = 0$ there does not create problematic linear dependences.

\begin{prop}\label{prop737}Let $\lambda$ and $\lambda'$ be partitions.
\begin{enumerate}[(i)]\setlength\itemsep{-.5em}

\item Let $\beta \in U_\lambda(n)$ and $\beta' \in U_{\lambda'}(n)$.  If $s_\lambda(\beta; x) = s_{\lambda'}(\beta'; x)$, then $\lambda = \lambda'$. \label{prop737.1}

\item Let $\pi \in S_n^\lambda$ and let $\pi' \in S_n^{\lambda'}$.  If $d_\lambda(\pi; x) = d_{\lambda'}(\pi'; x)$, then $\lambda = \lambda'$ and $\pi = \pi'$. \\ Hence $d_\lambda(\pi; x) \equiv d_{\lambda'}(\pi'; x)$. \label{prop737.2} \end{enumerate} \end{prop}

\begin{proof}  Let $T^0_\lambda$ denote the unique minimal element of $\mathcal{T}_\lambda$.  Note that $\Theta(T_\lambda^0) = \lambda$.  Clearly $T^0_\lambda \in \mathcal{S}_\lambda(\beta)$ and $T^0_\lambda \leq Y_\lambda(\pi)$.  Note that if $T, T' \in \mathcal{T}_\lambda$ are such that $T < T'$, then when $P(n)$ is ordered lexicographically from the left we have $\Theta(T) > \Theta(T')$.  So when $T^0_\lambda$ is in a subset of $T_\lambda$, it is the unique tableau in that subset that attains the lexicographic maximum of the contents in $P(n)$ for that subset.  Since $T^0_\lambda$ is a key, Fact \ref{fact420}(\ref{fact420.3}) gives $S(T^0_\lambda) = T^0_\lambda$.  So $S(T^0_\lambda) \leq Y_\lambda(\pi)$, and we have $T^0_\lambda \in \mathcal{D}_\lambda(\pi)$.  We can now see that $s_\lambda(\beta; x) = s_{\lambda'}(\beta'; x)$ and $d_\lambda(\pi; x) = d_{\lambda'}(\pi'; x)$ each imply that $\lambda = \lambda'$.  By Fact \ref{fact420}(\ref{fact420.5}), we know that $Y_\lambda(\pi)$ is the unique maximal element of $\mathcal{D}_\lambda(\pi)$.  So $Y_\lambda(\pi)$ is the unique tableau in $\mathcal{D}_\lambda(\pi)$ that attains the lexicographic minimum of the contents for $\mathcal{D}_\lambda(\pi)$.  By Fact \ref{fact420}(\ref{fact420.1}), since $Y_\lambda(\pi)$ is a key it is the unique tableau in $\mathcal{T}_\lambda$ with its content.  So $d_\lambda(\pi; x) = d_\lambda(\pi'; x)$ implies $Y_\lambda(\pi) \in \mathcal{D}_\lambda(\pi^\prime)$ and $Y_\lambda(\pi^\prime) \in \mathcal{D}_\lambda(\pi)$.  Hence $Y_\lambda(\pi) = Y_\lambda(\pi')$, which implies $\pi = \pi^\prime$.  \end{proof}

We now compare row bound sums to Demazure polynomials.  The first two parts of the following ``sufficient'' theorem quickly restate most of Parts (i) and (ii) of Theorem \ref{theorem721} in the current context, and the third similarly recasts Part (iv).

\begin{thm}\label{theorem737.1}Let $\lambda$ be a partition.
\begin{enumerate}[(i)]\setlength\itemsep{-.5em}

\item If $\eta \in UGC_\lambda(n)$, then $\Pi_\lambda[\Delta_\lambda(\eta)] =: \pi \in S_n^{\lambda\text{-}312}$ and $s_\lambda(\eta;x) \equiv d_\lambda(\pi;x)$. \label{theorem737.1.1}

\item If $\pi \in S_n^{\lambda\text{-}312}$, then $\Psi_\lambda(\pi) =: \gamma \in UG_\lambda(n)$ and $d_\lambda(\pi;x)  \equiv s_\lambda(\gamma;x)$. \label{theorem737.1.2}

\item Every flag Schur polynomial is identical to a uniquely determined Demazure polynomial and every $\lambda$-312-avoiding Demazure polynomial is identical to a uniquely determined flag Schur polynomial. \label{theorem737.1.3} \end{enumerate} \end{thm}

Next we obtain \emph{necessary} conditions for having equality between a row bound sum and a Demazure polynomial: we see that using the weaker notion of equality between mere polynomials does not lead to any new coincidences.

\begin{thm}\label{theorem737.2}Let $\lambda$ and $\lambda'$ be partitions.  Let $\beta$ be an upper $\lambda$-tuple and let $\pi$ be a $\lambda'$-permutation.  Suppose $s_\lambda(\beta; x) = d_{\lambda'}(\pi; x)$.  Then $Q_\lambda(\beta) = Y_{\lambda'}(\pi)$.  Hence $\lambda = \lambda'$ and $\Delta_\lambda(\beta) = \Psi_\lambda(\pi)$.  Here $\pi$ is $\lambda$-312-avoiding and $\Delta_\lambda(\beta)$ is a gapless $\lambda$-tuple (and so $\beta \in UGC_\lambda(n)$).  Hence the only row bound sums that arise as Demazure polynomials are the flag Schur polynomials.  We have $s_\lambda(\beta; x) \equiv d_{\lambda'}(\pi; x)$.  The row bound sum $s_\lambda(\beta; x)$ is identical to the flag Schur polynomial $s_\lambda(\Phi_\lambda[\Delta_\lambda(\beta)];x)$.  \end{thm}

\begin{proof}\noindent Reasoning as in the first part of the proof of Proposition \ref{prop737} implies $\lambda = \lambda'$.  Since $\mathcal{S}_\lambda(\beta) = [Q_\lambda(\beta)]$ by Proposition \ref{prop623.8}(\ref{prop623.8.1}), the tableau $Q_\lambda(\beta)$ is the unique tableau in $\mathcal{S}_\lambda(\beta)$ that attains the lexicographic minimum of the contents for $\mathcal{S}_\lambda(\beta)$.  Since the analogous remark was made in the proof of Proposition \ref{prop737} for $Y_\lambda(\pi) \in \mathcal{D}_\lambda(\pi)$, we must have $\Theta[Q_\lambda(\beta)] = \Theta[Y_\lambda(\pi)]$.  But $Y_\lambda(\pi)$ is the unique tableau in $\mathcal{T}_\lambda$ with its content.  So we must have $Q_\lambda(\beta) = Y_\lambda(\pi)$.  As for Theorem \ref{theorem721}, this implies $\Delta_\lambda(\beta) = \Psi_\lambda(\pi)$.  Since $\mathcal{D}_\lambda(\pi) \subseteq [Y_\lambda(\pi)]$, we have $\mathcal{D}_\lambda(\pi) \subseteq [Q_\lambda(\beta)] = \mathcal{S}_\lambda(\beta)$.  Suppose that $\pi$ is $\lambda$-312-containing.  Then Theorem \ref{theorem520} says $\mathcal{D}_\lambda(\pi) \neq [Q_\lambda(\beta)]$.  So $\mathcal{D}_\lambda(\pi) \subset \mathcal{S}_\lambda(\beta)$.  This implies that $d_\lambda(\pi; x) \neq s_\lambda(\beta; x)$, a contradiction.  So $\pi$ must be $\lambda$-312-avoiding.  Therefore $\Psi_\lambda(\pi) =: \gamma \in UG_\lambda(n)$.  Use Proposition \ref{prop623.1} for the ``only row bound sums that can arise'' statement.  Theorem \ref{theorem721}(\ref{theorem721.2}) now says that $\mathcal{D}_\lambda(\pi) = \mathcal{S}_\lambda(\gamma)$.  And $\gamma = \Delta_\lambda(\beta)$ gives $\mathcal{S}_\lambda(\gamma) = \mathcal{S}_\lambda(\beta)$.  Since $\gamma \in UG_\lambda(n)$, we can form the $\lambda$-floor flag $\Phi_\lambda[\Delta_\lambda(\beta)] \sim \beta$. \end{proof}

By using the relating of row bound sums to Demazure polynomials in this theorem, we can extend what was said in Proposition \ref{prop737}(\ref{prop737.1}) concerning accidental equalities between row bound sums.  There we learned that $s_\lambda(\beta;x) = s_\lambda(\beta';x)$ forced $\lambda = \lambda'$.  So here we need consider only one partition:

\begin{cor}\label{newcor737}Let $\lambda$ be a partition.
\begin{enumerate}[(i)]\setlength\itemsep{-.5em}

\item Let $\beta \in U_\lambda(n)$ and $\eta \in UGC_\lambda(n)$.  If $s_\lambda(\beta;x) = s_\lambda(\eta;x)$ then $\beta \sim \eta$.  Hence $\beta \in UGC_\lambda(n)$ and $s_\lambda(\beta;x) \equiv s_\lambda(\eta;x)$. \label{newcor737.1}

\item The partitionings of $UGC_\lambda(n)$ into the equivalence class intervals in Proposition \ref{prop608.4}(\ref{prop608.4.2}) give a complete description of the indexing ambiguity and of non-equality for gapless core Schur polynomials. \label{newcor737.2}

\item More specifically, the analogous statement for $UF_\lambda(n)$ and flag Schur polynomials follows from Proposition \ref{prop608.4}(\ref{prop608.4.3}). \label{newcor737.3} \end{enumerate} \end{cor}

\noindent Parts (ii) and (iii) could have been derived from Theorem \ref{theorem737.1}.

\begin{proof}Create $\Pi_{\lambda}[\Delta_{\lambda}(\eta)] =: \pi \in S_n^{\lambda\text{-}312}$ from $\Delta_{\lambda}(\eta) \in UG_\lambda(n)$.  Apply Theorem \ref{theorem737.1}(\ref{theorem737.1.2}) to obtain $d_{\lambda}(\pi; x) \equiv s_\lambda(\Delta_\lambda(\eta);x) \equiv s_\lambda(\eta;x)$.  Then apply Theorem \ref{theorem737.2} to $s_\lambda(\beta; x) = d_{\lambda}(\pi; x)$ to obtain $s_\lambda(\beta; x) \equiv d_\lambda(\pi; x)$.  So $s_\lambda(\beta; x) \equiv s_{\lambda}(\eta; x)$, which implies $\beta \sim \eta$ and $\beta \in UGC_\lambda(n)$.  \end{proof}

We do not know if it is possible to rule out accidental coincidences between all pairs of row bound sums:

\begin{prob}\label{prob14.5}Find $n \geq 1$, a partition $\lambda$, and $\beta, \beta^\prime \in U_\lambda(n) \backslash UGC_\lambda(n)$ such that $s_\lambda(\beta;x) = s_\lambda(\beta^\prime;x)$ but $\Delta_\lambda(\beta) \neq \Delta_\lambda(\beta^\prime)$. \end{prob}

Improving upon Equation 13.1 and Corollary 14.6 of \cite{PS}, in \cite{PW4} we will give a ``maximum efficiency'' determinant expression for the Demazure polynomials $d_\lambda(\pi;x)$ with $\pi \in S_n^{\lambda\text{-}312}$.

\section{Projecting and lifting the notion of 312-avoidance}\label{824}

In Propositions \ref{prop824.2} and \ref{prop824.4} we use the six maps $\Psi, \Pi, \Psi_R, \Pi_R, \Delta_R$, and $\Phi_R$ that we developed for other purposes to relate the notion of $R$-312-avoidance to that of 312-avoidance.  Some of the applications of these maps ``sort'' the entries of the $R$-tuples within their carrels.  The proofs of the propositions in this section appeared in Section 7 of \cite{PW2}.

If $\sigma \in S_n$ is 312-avoiding, it is easy to see that its $R$-projection $\bar{\sigma} \in S_n^R$ is $R$-312-avoiding.  Let $\pi \in S^R_n$ be $R$-312-avoiding.  Is it the $R$-projection $\bar{\sigma}$ of some 312-avoiding permutation $\sigma \in S_n$?  The following procedure for constructing an answer to this question can be naively developed, keeping in mind Fact \ref{fact320.3}(\ref{fact320.3.3}):  Form the $R$-rightmost clump deleting chain $B$ associated to $\pi$.  Set $\sigma_i := \pi_i$ on the first carrel $(0, q_1]$.  Let $h \in [r]$.  Let $s \geq 0$ be the number of elements of $B_{h+1} \backslash B_{h}$ that are less than $\max(B_{h}) =: m$.  List these elements in decreasing order to fill the left side $(q_h, q_h+s]$ of the $(h+1)^{st}$ carrel $(q_h, q_{h+1}]$ of $\sigma$.  Fill the right side $(q_h+s, q_{h+1}]$ of this carrel of $\sigma$ by listing the other $t := p_{h+1} - s$ elements of $B_{h+1} \backslash B_{h}$ in increasing order.  Part (ii) of the following result refers to the ``length'' of a permutation in the sense of Proposition 1.5.2 of \cite{BB}.

\begin{prop}\label{prop824.1}Suppose $\pi \in S^R_n$ is $R$-312-avoiding.
\begin{enumerate}[(i)]\setlength\itemsep{-.5em}

\item The permutation $\sigma \in S_n$ constructed here is 312-avoiding and $\bar{\sigma} = \pi$. \label{prop824.1.1}

\item This $\sigma$ is the unique minimum length 312-avoiding lift of $\pi$. \label{prop824.1.2} \end{enumerate} \end{prop}

This lifting process can also be described using three existing maps.  To pass from the ``degenerate'' $R$-world to the full $R = [n-1]$ world of ordinary permutations, for the second equality below we use the map $\Pi$.  This produces a final output of a permutation from the given $R$-permutation input.  We will use the following result to derive a weaker version of our Theorem \ref{theorem737.1}(\ref{theorem737.1.2}) from Theorem 14.1 of \cite{PS}:

\begin{prop}\label{prop824.2}Suppose $\pi \in S_n^R$ is $R$-312-avoiding.  Let $\sigma \in S_n$ be the minimum length 312-avoiding lift of $\pi$.  Then $\Delta_R[\Psi(\sigma)] = \Psi_R(\pi)$ and so $\sigma = \Pi[\Phi_R(\Psi_R(\pi))]$.  \end{prop}

We further consider an $R$-312-avoiding permutation $\pi$ and its associated $R$-rightmost clump deleting chain, again keeping in mind the picture provided by Fact \ref{fact320.3}(\ref{fact320.3.3}).  We want to describe all 312-avoiding lifts $\sigma^\prime$ of $\pi$.  Let $h \in [r]$.  As in Section \ref{320}, let $B_{h+1} =: L_1 \cup L_2 \cup ... \cup L_f$ decompose $B_{h+1}$ into clumps for some $f \geq 1$.  Restating the clump deleting condition in Section \ref{320}, we take $e \in [f]$ to be maximal such that $L_e \cap B_h \neq \emptyset$ and $B_{h+1} \backslash B_h \supseteq L_{e+1} \cup ... \cup L_f$.  The $s$ elements of $B_{h+1} \backslash B_h$ that are smaller than $m$ are in the clump $L_e$.  It is possible that some elements $m+1, m+2, ...$ from $B_{h+1} \backslash B_h$ are also in $L_e$.  Set $m^\prime := \max(L_e)$ and $s^\prime := | L_e \backslash B_h |$.  Since $\pi$ is $R$-increasing, when $s^\prime > s$ we have $\pi_{q_h + s + 1} = m+1, ... , \pi_{q_h+s^\prime} = m^\prime$ with $m^\prime - m = s^\prime - s$.  So then $\pi$ contains this staircase within the subinterval $(q_h+s, q_h + s^\prime]$ of $(q_h, q_{h+1}]$.  In any case we refer to the cohort $L_e \backslash B_h$ on $(q_h, q_h + s^\prime]$ as the (possibly empty) \emph{subclump} $L_e^\prime$ of $L_e$.

\begin{fact}With respect to the entities introduced above for $h \in [r]$:  Corresponding to the clumps $L_{e+1}, ... , L_f$ of $B_{h+1}$ there are respective staircases of $\pi$ within $(q_h, q_{h+1}]$.  When $s^\prime > s$ there is also a staircase of $\pi$ within $(q_h+s, q_h+s^\prime]$.  The supports of these staircases ``pave'' $(q_h+s,q_{h+1}]$.  An analogous statement with no subclump holds for $\pi$ on the first carrel $(0,q_1]$.  \end{fact}

\begin{prop}\label{prop824.3}Suppose $\pi \in S_n^R$ is $R$-312-avoiding.  Let $\sigma^\prime$ be a 312-avoiding lift of $\pi$.  In terms of the entities above, this lift $\sigma^\prime$ may be obtained from the minimum length 312-avoiding lift $\sigma$ of $\pi$ as follows:  Let $h \in [r]$.  For each of the clumps $L_{e+1}, ... , L_f$ of $B_{h+1}$, its entries in $\sigma$ may be locally rearranged on its support in any 312-avoiding fashion when forming $\sigma^\prime$.  The entries for the subclump $L_e^\prime$ may be locally rearranged on $(q_h, q_h + s^\prime]$ in any 312-avoiding fashion provided that its entries less than $m$ remain in decreasing order.  The entries for each of the clumps of $B_1$ may be locally rearranged as for $L_{e+1}, ... , L_f$.  Conversely, any such rearrangement of the entries of $\sigma$ produces a 312-avoiding lift of $\pi$.  \end{prop}

We will use the following result to derive Theorem 14.1 of \cite{PS} from our Theorem \ref{theorem737.1}(\ref{theorem737.1.2}):

\begin{prop}\label{prop824.4}Suppose $\pi \in S_n^R$ is $R$-312-avoiding.  Let $\sigma^\prime \in S_n$ be a 312-avoiding lift of $\pi$.  Then $\Delta_R[\Psi(\sigma^\prime)] = \Psi_R(\pi)$ and so $\pi = \Pi_R[\Delta_R[\Psi(\sigma^\prime)]]$.  \end{prop}

\section{Further remarks}\label{824B}

Table 16.1 of \cite{PW2} summarizes our results concerning the equality of the polynomials and the identicality of the generating functions associated to our tableau sets; there $\lambda$ and $\lambda'$ are partitions.  The justifications for the entries in that table appear in the first paragraph of Section 16.

Let $\lambda$ be a partition.  The equivalence classes $\langle \beta \rangle_{\approx_\lambda}$ of tableau sets $\mathcal{S}_\lambda(\beta)$ for $\beta \in U_\lambda(n)$ can be indexed by the elements of $UI_\lambda(n)$ according to Proposition \ref{prop623.4}(\ref{prop623.4.1}).  We know that $| UI_\lambda(n) | = {n \choose R_\lambda} =: {n \choose \lambda}$.  The Demazure tableau sets $\mathcal{D}_\lambda(\pi)$ are indexed by $\pi \in S_n^\lambda$ by Fact \ref{fact420}(\ref{fact420.6}), and we know $| S_n^\lambda | = {n \choose R_\lambda} = {n \choose \lambda}$.  By Theorem \ref{theorem737.2} and Fact \ref{fact420}(\ref{fact420.6}), we see that there are ${n \choose \lambda} - C_n^\lambda$ row bound sum generating functions (and at most this many row bound sum polynomials) that cannot equal any Demazure polynomial.  These two statements also imply that there are ${n \choose \lambda} - C_n^\lambda$ Demazure polynomials that cannot be equal to row bound sums as polynomials.  Can this mysterious coincidence be explained by an underlying phenomenon?

\begin{prob}\label{prob16.1}Let $\lambda$ be a partition.  Set $J := [n-1] \backslash R_\lambda$.  Is there a non-$T$-equivariant deformation of the Schubert varieties in the $GL(n)$ flag manifold $G/P_J$ that bijectively moves the torus characters $d_\lambda(\pi;x)$ to the $s_\lambda(\alpha;x)$ for $\pi \in S_n^\lambda$ and $\alpha \in UI_\lambda(n)$ with exactly $C_n^R$ fixed points, namely $d_\lambda(\sigma;x) = s_\lambda(\gamma;x)$ for $\sigma \in S_n^{\lambda\text{-}312}, \gamma \in UG_\lambda(n)$, and $\Psi_\lambda(\sigma) = \gamma$?  \end{prob}

\noindent Finding such a bijection would give a negative answer to Problem \ref{prob14.5}, since it would rule out accidental polynomial equalities among all row bound sums.  To get started, first compute the dimensions $|\mathcal{D}_\lambda(\pi)|$ and $|\mathcal{S}_\lambda(\alpha)|$ for all $\pi \in S_n^\lambda$ and $\alpha \in UI_\lambda(n)$ for some small partition $\lambda$.  Use this data to propose a guiding bijection from $UI_\lambda(n)$ to $S_n^\lambda$ that extends our bijection $\Pi_\lambda : UG_\lambda(n) \rightarrow S_n^{\lambda\text{-}312}$.

Working with an infinite number of variables $x_1, x_2, ...$, Reiner and Shimozono studied \cite{RS} coincidences between ``skew'' flag Schur polynomials and Demazure polynomials in their Theorems 23 and 25.  To indicate how those statements are related to our results, we consider only their ``non-skew'' flag Schur polynomials and specialize those results to having just $n$ variables $x_1, ... , x_n$.  Then their key polynomials $\kappa_\alpha(x)$ are indexed by ``(weak) compositions $\alpha$ (into $n$ parts)''.  The bijection from our pairs $(\lambda, \pi)$ with $\lambda$ a partition and $\pi \in S_n^\lambda$ to their compositions $\alpha$ that was noted in Section 3 of \cite{PW1} is indicated in the sixth paragraph of the Appendix to that paper:  Let $\pi \in S_n^\lambda$.  After creating $\alpha$ via $\alpha_{\pi_i} := \lambda_i$ for $i \in [n]$, here we write $\pi.\lambda := \alpha$.  Under this bijection the Demazure polynomial $d_\lambda(\pi;x)$ of \cite{PW1} and their key polynomial $\kappa_\alpha(x)$ are defined by the same recursion.  Reiner and Shimozono characterized the coincidences between the $s_\lambda(\phi;x)$ for $\phi \in UF_\lambda(n)$ and the $d_{\lambda'}(\pi;x)$ for $\pi \in S_n^{\lambda'}$ from the perspectives of both the flag Schur polynomials and the Demazure polynomials.  To relate the index $\phi$ to the index $\pi \in S_n^{\lambda'}$, their theorems refer to the tableau we denote $Q_\lambda(\phi)$.  Part (i) of the following fact extends the sixth paragraph of the Appendix of \cite{PW1}.  Part (ii) can be confirmed with Proposition \ref{prop623.8}(\ref{prop623.8.2}), Proposition \ref{prop604.4}(\ref{prop604.4.1}) and Lemma \ref{lemma340.1}.

\begin{fact}\label{fact824B}Let $\pi \in S_n^\lambda$.  Let $\phi \in UF_\lambda(n)$. \begin{enumerate}[(i)]\setlength\itemsep{-.5em}

\item The content $\Theta[Y_\lambda(\pi)]$ of the key of $\pi$ is the composition that has the unique \\ decomposition $\pi.\lambda$. \label{fact824B.1}

\item The tableau $Q_\lambda(\phi)$ is a $\lambda$-key $Y_\lambda(\sigma)$ for a uniquely determined $\sigma \in S_n^\lambda$. \label{fact824B.2} \end{enumerate} \end{fact}

From the perspective of flag Schur polynomials, in our language their Theorem 23 first said that every $s_\lambda(\phi;x)$ arises as a $d_{\lambda'}(\pi;x)$ for at least one pair $(\lambda', \pi)$ with $\lambda'$ a partition and $\pi \in S_n^{\lambda'}$.  Second, that $d_{\lambda'}(\pi;x)$ must be the Demazure polynomial for which $\pi.\lambda' = \Theta[Q_\lambda(\phi)]$.  Their first statement appears here as a weaker form of the first part of Theorem \ref{theorem737.1}(\ref{theorem737.1.3}).  The fact above can be used to show that their second (uniqueness) claim is equivalent to the first (and central) ``necessary'' claim $Q_\lambda(\phi) = Y_{\lambda'}(\pi)$ in our Theorem \ref{theorem737.2} that is produced by taking $\beta := \phi$.

From the other perspective, their Theorem 25 put forward a characterization for a Demazure polynomial $d_\lambda(\pi;x)$ that arises as a flag Schur polynomial $s_{\lambda'}(\phi;x)$ for some $\phi \in UF_{\lambda'}(n)$.  This characterization is stated in terms of a flag $\phi(\alpha)$ that is specified by a recipe to be applied to a composition $\alpha$;  this is given before the statement of Theorem 25.  Let $(\lambda, \pi)$ be the pair corresponding to $\alpha$.  It appears that the recipe for $\phi(\alpha)$ should have ended with `having size $\lambda_i$' instead of `having size $\alpha_i$'; we take this fix for granted for the remainder of the discussion.  But no recipe of this form can be completely useful for general partitions $\lambda$ since any $\lambda$-tuple $\phi(\alpha)$ produced will be constant on the carrels of $[n]$ determined by $\lambda$.  So first we consider only strict $\lambda$.  Here it can be seen that their $\phi(\alpha)$ becomes our $\Psi(\pi) =: \psi$.  Thus their $T_{\lambda(\alpha),\phi(\alpha)}$ is our $Q(\psi)$, and so the condition $Key(\alpha) = T_{\lambda(\alpha),\phi(\alpha)}$ translates to $Y(\pi) = Q(\psi)$.  Following the statement of Theorem \ref{theorem340}, we noted that the converse of its Part (ii) held when $\lambda$ is strict.  Using Proposition \ref{prop623.8}(\ref{prop623.8.2}), Theorem \ref{theorem340}(\ref{theorem340.2}), and that fact we see that this $Y(\pi) = Q(\psi)$ condition is equivalent to requiring $\pi \in S_n^{312}$.  So when $\lambda$ is strict the two directions of Theorem 25 appear in this paper as parts of Theorem \ref{theorem737.1}(\ref{theorem737.1.2}) and Theorem \ref{theorem737.2}.

Now consider Theorem 25 for general $\lambda$.  Its hypothesis $\kappa(\alpha) = S_{\lambda / \mu}(X_\phi)$ translates to $d_\lambda(\pi;x) = s_{\lambda'}(\phi;x)$.  In the necessary direction a counterexample to their condition $Key(\alpha) = T_{\lambda(\alpha),\phi(\alpha)}$, which translates to $Y_\lambda(\pi) = Q_\lambda(\phi(\alpha))$, is given by $\alpha = (1,2,0,1)$.  Their condition can be ``loosened up'' by replacing `$T_{\lambda(\alpha),\phi(\alpha)}$' with `$T_{\lambda',\phi}$', which translates to $Q_{\lambda'}(\phi)$.  This repaired version now gives the necessary condition $Y_\lambda(\pi) = Q_{\lambda'}(\phi)$, which is the central claim of Theorem \ref{theorem737.2}.  Turning to the sufficient direction:  From looking at the $\lambda$-row end list of $Y_\lambda(\pi)$, it can be seen that the $nn$-tuple $\phi(\alpha) =: \phi$ is in $UF_\lambda(n) \subseteq UGC_\lambda(n)$ as well as being constant on the carrels of $\lambda$.  Suppose that their condition $Y_\lambda(\pi) = Q_\lambda(\phi)$ is satisfied, and set $\Delta_\lambda(\phi) =: \gamma \in UG_\lambda(n)$.  Then $Y_\lambda(\pi) = Q_\lambda(\gamma)$, and Theorem \ref{theorem340} gives $\pi \in S_n^{\lambda\text{-}312}$.  Then Theorem \ref{theorem737.1}(\ref{theorem737.1.2}) implies that $d_\lambda(\pi;x) = s_\lambda(\phi;x)$, which confirms this part of Theorem 25.  However, the set of cases $(\lambda,\pi)$ that are produced by this sufficient condition is smaller than that produced by the $\lambda$-312-avoiding sufficient condition:  It can be seen that each index $\gamma$ produced above has only a single critical entry in each carrel of $\lambda$, while the general indexes $\gamma'$ that can arise for such coincidences range over all of the larger set $UG_\lambda(n)$.  Further remarks appear in the fifth paragraph of Section 15 of \cite{PW2}.

We prepare to discuss a related result \cite{PS} of Postnikov and Stanley.  Let $\pi \in S_n^\lambda$.  Our definitions of the $\lambda$-chain $B$ and the $\lambda$-key $Y_\lambda(\pi)$ can be extended to all of $S_n$ so that $Y_\lambda(\sigma') = Y_\lambda(\pi)$ exactly for the $\sigma' \in S_n$ such that $\bar{\sigma'} = \pi$.  Then our definition of Demazure polynomial can be extended from $S_n^\lambda$ to $S_n$ such that $d_\lambda(\sigma';x) = d_\lambda(\pi;x)$ for exactly the same $\sigma'$.  Their paper used this ``looser'' indexing for the Demazure polynomials.

In their Theorem 14.1, Postnikov and Stanley stated a sufficient condition for a coincidence to occur from the perspective of Demazure polynomials:  If $\pi \in S_n$ is 312-avoiding, then $d_\lambda(\pi;x) = s_\lambda(\phi;x)$ for a certain $\phi \in UF_\lambda(n)$.  After noting that this theorem followed from Theorem 20 of \cite{RS}, they provided their own proof of it.  Their bijective recipe for forming $\phi$ from $\pi$ was complicated.  Their inverse for this bijection is our inverse map $\Pi$ of Proposition 6.3(ii) of \cite{PW2}, which is the full $R = [n-1]$ case of Proposition \ref{prop320.2}(\ref{prop320.2}) here.  This map takes upper flags to 312-avoiding permutations.  Since the inverse of the inverse of a bijection must be the bijection, from that proposition it follows that their recipe for $\phi$ must be the restriction of our $\Psi$ to $S_n^{312}$.  The following result uses the machinery provided by our maps of $n$-tuples in Propositions \ref{prop824.2} and \ref{prop824.4} to prove that their theorem is equivalent to a weaker version of one of ours:

\begin{thm}\label{theorem824.5}Theorem 14.1 of \cite{PS} is equivalent to our Theorem \ref{theorem737.1}(\ref{theorem737.1.2}), once `$\equiv$' in the latter result has been replaced by `$=$'.  \end{thm}

\noindent The proof is in Section 15 of \cite{PW2}.  To convert their index $\sigma' \in S_n^{312}$ for a Demazure polynomial to an index for a flag Schur polynomial, Postnikov and Stanley set $\varphi' := \Psi(\sigma')$.  For one such $\sigma'$, our unique corresponding element of $S_n^{\lambda\text{-}312}$ is $\pi := \bar{\sigma'}$.  Let $\sigma$ be the minimum length $\lambda$-312-avoiding lift of $\pi$, and let $\sigma''$ be any other such lift.  We work with $\Psi_\lambda(\pi) =: \gamma \in UG_\lambda(n)$ and $s_\lambda(\gamma;x)$.  As they apply $\Psi$ to various $\sigma''$, they produce various upper flags $\varphi''$.  By Proposition \ref{prop824.4} we see $\varphi'' \sim \gamma$.  By Proposition \ref{prop824.2} it can be seen that our ``favored'' $\Psi(\sigma)$ is the $\lambda$-floor flag $\Phi_\lambda[\gamma] =: \tau$ that is the unique minimal upper flag such that $s_\lambda(\tau;x) = s_\lambda(\varphi'';x) = d_\lambda(\sigma'';x)$.

Gelfand-Zetlin (GZ) patterns are alternate forms of tableaux.  Kogan used certain subsets of these, which formed integral convex polytopes, to describe some Demazure polynomials.  In our predecessor paper \cite{PW3} that obtained the convexity Theorem \ref{theorem520}, while making some geometric remarks we noted that in Corollary 15.2 of \cite{PS} Postnikov and Stanley had shown that the Demazure polynomials described by Kogan are the 312-avoiding Demazure polynomials.  For strict $\lambda$, Kiritchenko, Smirnov, and Timorin vastly generalize that result in Theorem 1.2 of \cite{KST} by expressing any Demazure polynomial as the sum of the generating functions for the integral points lying on certain faces of the GZ polytope for $\lambda$.  In their introduction those authors note that only one face is used for this exactly when the indexing permutation is 132-avoiding; this is their equivalent to our 312-avoiding after taking into account their conjugation by $w_0$.  This observation led the referee for this paper to ask if the integral convexity phenomenon of our Theorem \ref{theorem520} is paralleled by an integral convexity phenomenon for GZ patterns.  We do not know of any relationship between integral convexity for a general set of semistandard tableaux and the integral convexity of the set of the corresponding GZ patterns.  Once one has descriptions of the set of tableaux for a flag Schur polynomials and of Kogan's set of patterns for a 312-avoiding permutation, it is easy to directly see that both form integral convex polytopes.  As for the converse convexity question for the sets of patterns used for general permutations, note that it is possible for a nontrivial union of faces of a ``small'' GZ polytope to be integrally convex:  Consider the GZ polytope for $\lambda = (1,2,3)$ shown in Figure 1 of \cite{KST}.  The Demazure polynomial for $w := s_2$ is obtained by summing over the integral points in the faces for $y=2$ and $x=z$.  This permutation is 132, which is not 132-avoiding.  The integral convex hull of the union of these two faces is the union itself.  So here the set of GZ patterns employed to compute a Demazure polynomial for a permutation is convex even though the permutation is not 132-avoiding.  But the desired converse is obtained by considering $\lambda' = m\lambda$ for $m$ sufficiently large, which is common in this area in algebraic geometry.  One can also conclude for large $m$ that the only candidates for coincidences with flag Schur polynomials are the 132-avoiding Demazure polynomials.

In Theorem 2.7.1 of \cite{St1}, Stanley used the Gessel-Viennot technique to give a determinant expression for a generating function for certain sets of $n$-tuples of non-intersecting lattice paths.  Then in his proof of Theorem 7.16.1 of \cite{St2}, he recast that generating function for some cases by viewing such $n$-tuples of lattice paths as tableaux.  After restricting to non-skew shapes and to a finite number of variables, his generating function becomes our row bound sum $s_\lambda(\beta;x)$ for certain $\beta \in U_\lambda(n)$.  Theorem 2.7.1 required that the pair $(\lambda, \beta)$ satisfies the requirement that Gessel and Viennot call \cite{GV} ``nonpermutable''.  There he noted that $(\lambda, \phi)$ is nonpermutable for every $\phi \in UF_\lambda(n)$;  this implicitly posed the problem of characterizing all $\beta \in U_\lambda(n)$ for which $(\lambda, \beta)$ is nonpermutable.  The ceiling map $\Xi_\lambda: UG_\lambda(n) \longrightarrow UF_\lambda(n)$ defined in Sections \ref{608} and \ref{604} can be extended to all of $U_\lambda(n)$.  For a given partition $\lambda$, the main result of \cite{PW4} says that $\beta \in U_\lambda(n)$ gives a nonpermutable $(\lambda, \beta)$ if and only if $\beta \in UGC_\lambda(n)$ and $\beta \leq \Xi_\lambda(\beta)$.  Hence we will again see that restricting consideration from all upper $\lambda$-tuples $\beta \in U_\lambda(n)$ down to at least the gapless core $\lambda$-tuples $\eta \in UGC_\lambda(n)$ enables saying something nice about the row bound sums $s_\lambda(\eta;x)$.  To compute $s_\lambda(\eta;x)$ for a given $\eta \in UGC_\lambda(n)$, the possible inputs for the Gessel-Viennot determinant are the $\eta' \in UGC_\lambda(n)$ such that $\eta' \sim \eta$ and $\eta' \leq \Xi_\lambda(\eta')$.  We will say that a particular such $\lambda$-tuple \emph{attains maximum efficiency} if the corresponding determinant has fewer total monomials among its entries than does the determinant for any other application of Theorem 2.7.1 to a $\beta \in U_\lambda(n)$ that produces $s_\lambda(\eta;x)$.  In \cite{PW4} we show that the gapless $\lambda$-tuple $\Delta_\lambda(\eta)$ attains maximum efficiency.

\section{Parabolic Catalan counts}\label{1800}

In the last section of our predecessor paper \cite{PW3} we listed a half-dozen ways in which the parabolic Catalan numbers $C_n^R$ arose in that paper, including as the number of gapless $R$-tuples.  Then we described how these numbers had recently independently arisen in the literature and in the Online Encyclopedia of Integer Sequences.  The following result lists several further occurrences of these numbers here; the section (or paper) cited at the beginning of each item points to the definition of the concept:

\begin{thm}\label{theorem18.1}Let $R \subseteq [n-1]$.  Write the elements of $R$ as $q_1 < q_2 < ... < q_r$.  Set $q_0 := 0$ and $q_{r+1} := n$.  Let $\lambda$ be a partition $\lambda_1 \geq \lambda_2 \geq ... \geq \lambda_n \geq 0$ whose shape has the distinct column lengths $q_r, q_{r-1}, ... , q_1$.  Set $p_h := q_h - q_{h-1}$ for $1 \leq h \leq r+1$.  The number $C_n^R =: C_n^\lambda$ of $R$-312-avoiding permutations is also equal to the number of:

\begin{enumerate}[(i)]\setlength\itemsep{-.5em}

\item Section \ref{300}:  flag $R$-critical lists. \label{theorem18.1.1}

\item Section \ref{300}:  $R$-canopy tuples $\kappa$, $R$-floor flags $\tau \in UFlr_R(n)$, $R$-ceiling flags $\xi \in UCeil_R(n)$. \label{theorem18.1.2}

\item Sections \ref{608} and \ref{623}:  the four collections of equivalence classes in $UGC_R(n) \supseteq UF_R(n)$ (or $UGC_\lambda(n)$ $\supseteq UF_\lambda(n)$) that are defined by the equivalence relations $\sim_R$ (or $\approx_\lambda$) respectively. \label{theorem18.1.3}

\item Section \ref{737}:  Demazure polynomials $d_\lambda(\pi;x)$ indexed by $\pi \in S_n^{\lambda\text{-}312}$, gapless core Schur polynomials $s_\lambda(\eta;x)$ for $\eta \in UGC_\lambda(n)$ that are distinct as polynomials, and flag Schur polynomials $s_\lambda(\phi;x)$ for $\phi \in UF_\lambda(n)$ that are distinct as polynomials.\label{theorem18.1.4}

\item Sections \ref{623}, \ref{737}:  Coincident pairs ($\mathcal{S}_\lambda(\beta), \mathcal{D}_\lambda(\pi)$) of sets of tableaux of shape $\lambda$ and coincident pairs ($s_\lambda(\beta;x), d_\lambda(\pi;x)$) of polynomials indexed by upper $\lambda$-tuples $\beta \in U_\lambda(n)$ and $\lambda$-permutations $\pi \in S_n^\lambda$. \label{theorem18.1.5}

\item Section \ref{824B}:  valid upper $\lambda$-tuple inputs to the Gessel-Viennot determinant expressions for flag Schur polynomials on the shape $\lambda$ that attain maximum efficiency. \label{theorem18.1.6} \end{enumerate} \end{thm}

\begin{proof}By Proposition \ref{prop320.2}(\ref{prop320.2.2}), the number of gapless $R$-tuples is $C_n^R$.  Use Corollary \ref{cor604.8} to confirm (i) and (ii).  Then use Corollary \ref{cor608.6} and Proposition \ref{prop623.2} for (iii).  Use Proposition \ref{prop737}(\ref{prop737.2}) and Corollary \ref{newcor737} to confirm (iv) and Theorems \ref{theorem721} and \ref{theorem737.2} for (v).  Part (vi) is Proposition 8.3 of \cite{PW4}.  \end{proof}

In \cite{PW3} we defined the \emph{total parabolic Catalan number $C_n^\Sigma$} to be $\sum C_n^R$, sum over $R \subseteq [n-1]$.

\begin{cor}\label{cor18.2}  For each $m \geq 1$, the total parabolic Catalan number $C_n^\Sigma$ is the number of flag Schur polynomials in $n$ variables on shapes with at most $n-1$ rows in which there are $m$ columns of each column length that is present.  \end{cor}

\noindent \textbf{Acknowledgements.}  Feedback from Vic Reiner encouraged us to complete our analysis of the related results in \cite{RS} and of their relationship to our results.  We thank David Raps for assistance in preparing the precursor paper \cite{PW2}, and the referee for alerting us to the existence of \cite{KST}.



\end{document}